\documentclass[reqno,a4paper]{article}
\usepackage{amssymb,amsmath,epsfig,graphics,psfrag,graphicx,color,cite,subfigure,fancyhdr,palatino}
\definecolor{lightblue}{rgb}{0.22,0.45,0.70}
\usepackage[colorlinks=true,breaklinks=true,linkcolor=lightblue,citecolor=lightblue]{hyperref}

\definecolor{darkred}{rgb}{0.82,0.15,0.20}
\definecolor{darkblue}{rgb}{0.82,0.15,0.12}

\numberwithin{equation}{section}
\numberwithin{figure}{section}
\numberwithin{table}{section}
\newcommand\cero{\boldsymbol{0}}

\newcommand\bH{\boldsymbol{H}}
\newcommand\bL{\boldsymbol{L}}
\newcommand\bI{\boldsymbol{I}}

\newcommand\bV{\boldsymbol{V}}

\newcommand\bPi{\boldsymbol{\Pi}}

\newcommand\beps{\boldsymbol{\varepsilon}}
\newcommand\ff{\boldsymbol{f}}
\newcommand\bb{\boldsymbol{b}}

\newcommand\nn{\boldsymbol{n}}

\newcommand\bsigma{\boldsymbol{\sigma}}

\newcommand\bu{\boldsymbol{u}}
\newcommand\bv{\boldsymbol{v}}

\newcommand\bx{\boldsymbol{x}}

\newcommand\RR{\mathbb{R}}
\newcommand\cT{\mathcal{T}}

\newcommand\bdiv{\mathop{\mathbf{div}}\nolimits}
\newcommand\vdiv{\mathop{\mathrm{div}}\nolimits}
\newcommand\bt{\boldsymbol{t}}

\newtheorem{remark}{Remark}[section]
\newtheorem{lemma}{Lemma}[section]
\newtheorem{theorem}{Theorem}[section]

\newtheorem{corollary}{Corollary}

\newenvironment{proof}{\noindent{\it Proof.}}{\hfill$\square$}

\allowdisplaybreaks

\begin{document}
\title{\Large {\bf Virtual element methods for the three-field formulation 
		of time-dependent linear poroelasticity}}

	\author{\normalsize Raimund B\"urger\thanks{
	CI$^{\,2}\!$MA and Departamento de Ingenier\'\i a Matem\' atica, 
		Universidad de Concepci\' on, Casilla 160-C, Concepci\' on, Chile. 
		Email: {\tt rburger@ing-mat.udec.cl}.},\quad 
	Sarvesh Kumar\thanks{Department of Mathematics,
		Indian Institute of Space
		Science and Technology, Trivandrum 695 547, India. 
		Email: {\tt sarvesh@iist.ac.in}.},\quad 
	David Mora\thanks{GIMNAP, Departamento de Matem\'atica,
		Universidad del B\'io-B\'io, Concepci\'on, Chile; and
		CI$^{\,2}\!$MA Universidad de Concepci\' on, Chile.
		Email: {\tt dmora@ubiobio.cl}.},\quad 
	Ricardo Ruiz-Baier\thanks{Mathematical Institute,
		University of Oxford,
		Woodstock Road, Oxford OX2 6GG, UK; and 
		Laboratory of Mathematical Modelling, 
		Institute for Personalised Medicine, Sechenov University, Moscow, Russian Federation. 
		Email: {\tt ruizbaier@maths.ox.ac.uk}.},\quad
	 Nitesh Verma\thanks{Department of Mathematics,
		Indian Institute of Space
		Science and Technology, Trivandrum 695 547, India.
		Email: {\tt niteshverma.16@res.iist.ac.in}.}
	}

\date{}
\maketitle

	\begin{abstract}
A virtual element discretisation for the numerical approximation of the three-field formulation of linear poroelasticity introduced in [R.\ Oyarz\'ua and R.\ Ruiz-Baier, Locking-free finite
		element methods for poroelasticity, {\it  SIAM J.\ Numer.\ Anal.}  \textbf{54} (2016)
		2951--2973] is proposed. The treatment is extended to include also the transient case. Appropriate poroelasticity projector operators are introduced and they assist in deriving energy bounds for the time-dependent discrete problem. Under standard assumptions on the computational domain,
		optimal   a priori error estimates are established. Furthermore, the accuracy of the method is verified numerically  through a set of computational tests.
	\end{abstract}
	
	
	
\maketitle

	\section{Introduction}
	
	The equations of linear poroelasticity describe the interaction between interstitial fluid flowing through deformable porous media. This problem, often referred to as Biot's consolidation,  has wide range of applications in diverse areas including biomechanics, groundwater management, oil extraction, earthquake engineering, or material sciences\cite{arega08,lee19,mauk03,moee13,peters02,sacco17}.
	
	A variety of numerical methods has been used to generate approximate solutions to the Biot consolidation problem. Modern examples include high-order finite differences\cite{gaspar02}, conforming finite elements\cite{aguilar08,murad96}, mixed finite element methods\cite{berger17,hong18},
	nodal and local discontinuous Galerkin methods\cite{hu17,riviere17},
	finite volume schemes\cite{asadi14,naumovich06}, and combined/hybrid
	discretisations\cite{coulet19,fu19,kumar19},   and we also
	point out Ref.~\cite{boffi16} where the authors present a polygonal discretisation
	based on hybrid high-order methods.
	These schemes are constructed using different formulations of the
	governing equations including primal and several types of mixed forms.
	
	In this paper we propose a virtual element method (VEM) using a three-field formulation of the time-dependent poromechanics equations. We base the development following the formulation proposed  in Refs.~\cite{lee17} and~\cite{oyarzua16} for the stationary Biot system and extend the discrete analysis  to include the quasi-steady case.
	We stress that this is not the first VEM formulation for the Biot equations,
	as Ref. \cite{coulet19} proposes a method that combines  VEM and finite volumes for
	the solid and fluid parts of the problem, respectively.
	
	Advantages of VEM include the relaxation of computing basis functions
	(of particular usefulness when dealing with high-order approximations),
	and the flexibility of computing solutions on general-shaped meshes
	(for instance, including non-convex elements).
	In addition, one works locally on polygonal elements, without the need of passing through a
	reference element, see e.g. Refs.~\cite{ahmad13,daveiga-b13,daveiga-e13,daveiga-g13,mora18}.
	This further simplifies the implementation of the building blocks
	of the numerical method. Polytopal meshes can be now generated with
	accurate tools such as  CD-adapto\cite{daveiga-div17}.
	
	Here we consider a pair of virtual elements for displacement and total pressure which is stable. This pair, introduced in
	Ref.~\cite{antonietti14},  can be regarded as a generalisation of the Bernardi-Raugel finite elements (piecewise linear elements enriched with bubbles normal to the faces for the displacement components, and piecewise constant approximations for total pressure, see e.g. Ref. \cite{girault86}).
	On the other hand, no compatibility between the spaces for total pressure and fluid pressure is needed. Therefore
	for the fluid pressure we employ the enhanced virtual element space
	from Refs.~\cite{abms20,daveiga-g13,vb2016}, which allows us  to
	construct a suitable projector onto piecewise linear functions.
	All this is restricted, for sake of simplicity, to the
	lowest-order 2D case, but one could extend the analysis to higher polynomial degrees
	and the 3D, for instance considering the discrete inf-sup stable pair from Ref. \cite{daveiga-div17}
	for the Stokes problem.
	The main difficulties in our analysis lie in the definition of an adequate projection operator that allows to treat the time-dependent
	problem. To handle this issue  we have combined Stokes-like and elliptic
	operators that constitute the new map, here named
	poroelastic projector. We derive stability for semi-discrete and fully-discrete approximations and establish
	the optimal convergence of the virtual element scheme in the natural norms.
	These bounds turn to be robust with respect
	to the dilation modulus of the deformable porous structure.
	A further advantage of the proposed virtual discretisation
	is that it combines primal and mixed virtual element spaces.
	In addition, this work can be seen as a stepping stone in the study of more complex
	coupled problems including interface poroelastic phenomena and multiphysics (see, for instance, Refs. \cite{adgmr20,gira11,vermaI}).

	We have arranged the contents of the paper as follows. Section \ref{sec:model} is devoted to the definition of the linear poroelasticity problem, and it also contains the precise definition of the continuous weak formulation using three fields, and presents a few preliminary results needed in the semi-discrete analysis as well. In Section~\ref{sec:VEapprox} we introduce the virtual element approximation in semi-discrete form. We specify the virtual element spaces, we identify the degrees of freedom, and derive appropriate estimates for the discrete bilinear forms. The a priori error analysis has been derived in Section~\ref{sec:estimates}, with the help of the newly introduced poroelastic projection operator. The implementation of the problem on different
	families of polygonal meshes is then discussed in Section~\ref{sec:results}, where
	we confirm the theoretical rates of convergence and produce some applicative tests to gain insight on the behaviour of the model problem.
	A summary and concluding remarks are collected in Section~\ref{sec:concl}.

	\section{Equations of time-dependent linear poroelasticity using total pressure}\label{sec:model}
	
	\subsection{Strong form of the governing equations}
	A deformable porous medium is assumed to occupy the domain~$\Omega$, where $\Omega$~is an open and
	bounded set in $\RR^2$ (simply for sake of notational convenience) with a  Lipschitz continuous boundary~$\partial \Omega$.
	The medium is
	composed by a mixture of
	incompressible grains forming a linearly elastic skeleton, as well as interstitial fluid. The mathematical description of this interaction between deformation and flow
	can be placed in the context of
	the classical Biot problem,  written as follows (see for instance, the exposition in Ref. \cite{showalter00}).
	In the absence of gravitational forces, and
	for a given body load $\bb(t):\Omega\to \RR^2$ and a volumetric source or sink $\ell(t):\Omega\to \RR$, one seeks,  for each time $t\in (0,t_{\mathrm{final}}]$,
	the vector of displacements of the  porous skeleton, $\bu(t):\Omega\to \RR^2$, and the pore pressure of the
	fluid, $p(t):\Omega\to\RR$, satisfying the mass conservation of the fluid content and momentum balance equations
	\begin{align*}
	\partial_t(c_0 p + \alpha\,\vdiv \bu ) - \frac{1}{\eta} \vdiv \bigl(\kappa(\bx) \nabla p \bigr) &= \ell,
	\\
	-\bdiv \bigl(  \lambda (\vdiv \bu)\bI + 2\mu \beps(\bu)- \alpha p\bI \bigr) & = \rho \bb \quad
	\text{in $\Omega\times(0,t_{\mathrm{final}}]$},
	\end{align*}
	where $\kappa(\bx)$~is the hydraulic conductivity of the porous medium (the mobility matrix, possibly anisotropic),  $\rho$~is the
	density of the solid material, $\eta$ is the constant viscosity of the
	interstitial fluid, $c_0$~is the constrained specific storage coefficient {(typically small and representing the amount of
		fluid that can be injected during an increase of pressure maintaining a constant bulk volume)}, $\alpha$~is the Biot-Willis
	consolidation parameter (typically close to one), and $\mu$ and~$\lambda$ are the shear and dilation moduli associated with the
	constitutive law of the solid structure. {The total stress
		\begin{align*}
		\bsigma = \lambda (\vdiv \bu)\bI + 2\mu \beps(\bu)- \alpha p \bI \end{align*}
		receives contribution from the effective mechanical stress of a Hookean elastic material, $\lambda (\vdiv \bu)\bI + 2\mu \beps(\bu)$, and
		the non-viscous fluid stress represented only by the pressure scaled with $\alpha$}.
	As in Refs. \cite{lee17,oyarzua16}, we consider here the volumetric part of the total stress $\psi$, hereafter called \emph{total pressure}, as one of the primary variables. And this allows us to rewrite the time-dependent problem as
	\begin{align} \label{eq:Biot} \begin{split}
	-\bdiv \bigl( 2\mu \beps(\bu)- \psi \bI \bigr)& = \rho \bb,  \\
	\biggl(c_0
	+\frac{\alpha^2}{\lambda}\biggr) \partial_t p -\frac{\alpha}{\lambda}  \partial_t \psi
	- \frac{1}{\eta} \vdiv(\kappa  \nabla p ) &= \ell,  \\
	\psi - \alpha p + \lambda \vdiv\bu &= 0 \quad \text{in $\Omega\times(0,t_{\mathrm{final}}]$},  \end{split}
	\end{align}
	which we endow with appropriate initial data ({for instance, assuming that the system is at rest})
	\begin{equation*}
	p(0) =0, \quad \bu(0) = \cero  \quad  \text{in $\Omega\times\{0\}$}
	\end{equation*}
	(which we can use to compute the initial condition for the total pressure $\psi(0)=0$) and boundary conditions in the following manner
	\begin{align}\label{bc:Gamma}
	\bu = \cero\quad \text{and} \quad \frac{\kappa}{\eta} \nabla p \cdot\nn &= 0 &  \text{on $\Gamma\times(0,t_{\text{final}}]$},\\
	\label{bc:Sigma}
	\bigl( 2\mu\beps(\bu) - \psi\,\bI \bigr) \nn = \cero \quad\text{and}\quad p&=0  &\text{on $\Sigma\times(0,t_{\text{final}}]$},
	\end{align}
	where the boundary $\partial\Omega = \Gamma\cup\Sigma$  is disjointly split into $\Gamma$ and $\Sigma$
	where we prescribe clamped boundaries and zero fluid normal fluxes; and zero (total) traction together with constant fluid pressure,
	respectively. Homogeneity of the boundary conditions is only assumed to simplify the exposition of the subsequent analysis.

	\subsection{Weak formulation}
	In order to obtain a weak form (in space) for \eqref{eq:Biot}, we define the function spaces
	$$\bV:=[H^1_{\Gamma}(\Omega)]^2, \; Q:= H^1_{\Sigma}(\Omega),\; Z:=L^2(\Omega).$$
	Multiplying \eqref{eq:Biot} by adequate test functions, integrating by parts (in space) whenever appropriate, and using the boundary conditions \eqref{bc:Gamma}-\eqref{bc:Sigma}, leads to the following variational problem: For a given $t>0$, find $\bu(t) \in \bV, p(t) \in Q, \psi(t) \in Z$ such that
	\begin{alignat}{5}
	&&   a_1(\bu,\bv)   &&                 &\;+&\; b_1(\bv,\psi)     &=&\;F(\bv)&\quad\forall \bv\in\bV, \label{weak-u}\\
	{\tilde{a}_2(\partial_t p,q)} &\; +&               &&      a_2(p,q)   &\;-&\;    b_2(q, \, \partial_t \psi) &=&\;G(q) &\quad\forall q\in Q, \label{weak-p}\\
	&&b_1(\bu,\phi)  &\;+\;& b_2(p,\phi)&\;-&\; a_3(\psi,\phi) &=&0 &\quad\forall\phi\in Z, \label{weak-psi}
	\end{alignat}
	where the bilinear forms
	$a_1:\bV\times\bV \to \RR$,
	$a_2:Q\times Q \to \RR$, $a_3: Z\times Z\to \RR$,
	$b_1:\bV \times Z \to \RR$, $b_2:Q \times Z\to \RR$,
	and linear functionals $F:\bV \to\RR$, $G:Q \to\RR$, are given by the following respective expressions:
	\begin{align} \label{bilinearforms} \begin{split}
	& a_1(\bu,\bv) := 2\mu \int_{\Omega} \beps(\bu):\beps(\bv),\quad b_1(\bv,\phi):= -\int_{\Omega}\phi\vdiv \bv, \\ &  F(\bv) := \int_{\Omega}  \rho\bb\cdot\bv, \quad
	\tilde{a}_2(p,q)  := \biggl( c_0 +\frac{\alpha^2}{\lambda}\biggr)
	\int_{\Omega}  p q ,\\ &  a_2(p,q)  :=\frac{1}{\eta}\int_{\Omega}\kappa \nabla p\cdot\nabla q,  \quad
	b_2(p,\phi):=  \frac{\alpha}{\lambda} \int_{\Omega} p\phi, \\ &
	a_3(\psi,\phi):=  \frac{1}{\lambda} \int_{\Omega} \psi\phi, \quad
	G(q) := \int_{\Omega} \ell \; q.
	\end{split} \end{align}
	
	\subsection{Properties of the bilinear forms and linear functionals}
	We now list the continuity, coercivity, and inf-sup conditions for the variational forms in  \eqref{bilinearforms}.
	These are employed in Ref. \cite{oyarzua16} to derive the well-posedness of the stationary form of \eqref{eq:Biot}.
	
	First we have the bounds
	\begin{align*}
	a_1(\bu, \bv) & \le 2\mu \| \beps(\bu) \|_0 \| \beps (\bv) \|_0 \le C \| \bu \|_1 \| \bv \|_1  & \quad \text{for all  $\bu, \bv \in \bV$,}  \\
	b_1(\bv, \phi)  & \le \| \vdiv \bv \|_0 \| \phi \|_0 \le C \| \bv \|_1  \| \phi \|_0 & \quad  \text{for all $\bv \in \bV$ and  $\phi \in Z$,} \\
	a_2(p, q) & \le  \frac{\kappa_{\max}}{ \eta} | p|_1 | q|_1 \le \frac{\kappa_{\max}}{ \eta} \| p \|_1 \| q \|_1  &\quad \text{for all $ p, q \in Q$, }\\
	b_2(q, \phi) & \le { \frac{\alpha}{\lambda} }\| q \|_0 \| \phi \|_0, \quad
	a_3(\psi, \phi)  \le \frac{1}{\lambda} \| \psi \|_0 \| \phi \|_0  &\quad \text{for all $q \in Q$ and $\psi, \phi \in Z$,}  \\
	F(\bv) & \le \rho \| \bb \|_0 \| \bv \|_1, \quad
	G(q) \; \le  \| \ell \|_0 \| q \|_0  & \quad  \text{for all $\bv \in \bV$ and $q \in Q$,}
	\end{align*}
	then the coercivity of the diagonal bilinear forms, i.e.,
	\begin{align*}
	a_1(\bv, \bv) & = 2\mu \| \beps(\bv) \|_0^2 \ge C \| \bv \|_1^2  &\quad \text{for all  $\bv \in \bV$,}  \\
	a_2(q, q) & \ge  \frac{\kappa_{\min}}{\eta} \| q \|_1^2 &\quad \text{for all  $q \in Q$,}  \\
	a_3(\phi, \phi) & = \frac{1}{\lambda} \| \phi \|_0^2 &\quad \text{for all  $\phi \in Z$,}
	\end{align*}
	and finally  satisfaction of the inf-sup condition, viz.\  there exists a constant $\beta >0$ such that  	\begin{align*}
	\sup_{\bv (\neq 0) \in \bV} \frac{b_1(\bv, \phi)}{\| \bv \|_1} \ge \beta \| \phi \|_0 \quad \text{for all $\phi \in Z$.}
	\end{align*}
	
	The solvability of the continuous problem is not the focus here, and we  refer to Ref.~\cite{showalter00}
	for the corresponding
	well-posedness and regularity results.
	
	\section{Virtual element approximation} \label{sec:VEapprox}
\subsection{Discrete spaces and degrees of freedom}
	In this section we construct a VEM associated with \eqref{weak-u}--\eqref{weak-psi}. We start denoting
	by $\{{\mathcal T}_h\}_h$ a sequence of partitions of the domain~$\Omega$
	into general polygons~$K$ (open and simply connected sets whose boundary $\partial K$ is a non-intersecting poly-line
	consisting of a finite number of straight line segments) having diameter $h_K$, and define as meshsize
	$h:=\max_{K\in{\mathcal T}_h}h_K$. By~$N^v_K$ we will denote the number of vertices in the polygon~$K$,
	$N^e_K$ will stand for the number of edges on $\partial K$,
	and $e$ a generic edge of $\mathcal{T}_h$. For all $e\in \partial K$, we denote by~$\boldsymbol{n}_K^e$
	the unit normal pointing outwards $K$, $\bt^e_K$ the unit tangent vector along~$e$ on~$K$, and $V_i$~represents the $i^{th}$ vertex of the polygon~$K$.
	
	As in Ref.~\cite{daveiga-b13} we need to assume  regularity of the
	polygonal meshes in the following sense: there exists
	$C_{{\mathcal T}}>0$ such that, for every $h$ and every $K\in {\mathcal T}_h$,
	the ratio between the shortest edge
	and $h_K$~is larger than $C_{{\mathcal T}}$; and
	$K\in{\mathcal T}_h$ is star-shaped with
	respect to every point within a  ball of radius~$C_{{\mathcal T}}h_K$.
	
	Denoting by $\mathbb{P}_k(K)$ the space of polynomials of degree up to $k$, defined locally on $K\in\cT_h$,
	we proceed to characterise the scalar energy projection operator $\Pi_{K}^{\nabla}: H^1(K) \rightarrow \mathbb{P}_1(K)$
	by the relations
	\begin{equation}\label{eq:proj-1}
	\bigl(\nabla (\Pi_{K}^{\nabla} q - q), \nabla r \bigr)_{0,K} = 0, \qquad P^0_K(\Pi_{K}^{\nabla} q - q)=0,
	\end{equation}
	valid for all $q \in H^1(K)$ and  $r \in \mathbb{P}_1(K)$, and where $( \cdot , \cdot )_{0,K}$
	denotes the $L^2$-product on $K$, and	
	$$ P^0_K(q):= \int_{\partial K} q\, \mathrm{d}s.$$
	
	If we now denote by $\mathcal{M}_k(K)$ the space of monomials of degree up to $k$, defined locally on $K\in\cT_h$,
	we can define, on each polygon $K \in \cT_h$, the local virtual element spaces for displacement,
	fluid pressure, and total pressure, as
	\begin{align} \label{VE-spaces} \begin{split}
	\bV_h(K) & := \Biggl\{ \bv_h \in [H^1(K)]^2 : \bv_h |_{\partial K} \in \mathbb{B}(\partial K),
	\\ &  \quad \qquad \begin{cases}
	- \Delta \bv_h - \nabla s = \boldsymbol{0} \text{ in } K, \\
	\vdiv \bv_h \in \mathbb{P}_0(K)
	\end{cases}\! \text{for some } s \in L^2_0(K)
	\Biggr\},  \\
	Q_h(K) &  := \bigl\{    q_h \in H^1(K) \cap C^0(\partial K): \;q_h|_e \in \mathbb{P}_1(e), \forall e \in \partial K, \\ & \quad  \qquad \Delta q_h|_K \in \mathbb{P}_1(K), \;   (\Pi_K^{\nabla} q_h - q_h, m_{\alpha})_{0,K} = 0\ \forall m_{\alpha} \in \mathcal{M}_1(K) \bigr\},  \\
	Z_h(K) & := \mathbb{P}_0(K),  \end{split} 	\end{align}
	where we define
	$$ \mathbb{B} (\partial K) := \bigl\{ \bv_h \in [C^0(\partial K)]^2: \bv_h|_e \cdot \bt^e_K \in \mathbb{P}_1(e), \bv_h|_e \cdot \nn^e_K \in \mathbb{P}_2(e), \forall e \in \partial K \bigr\}.$$
	
	It is clear from the above definitions that the dimension of~$\bV_h(K)$ is~$3N^e_K$, the
	dimension of~$Q_h(K)$ is~$N^v_K$, and that of~$Z_h(K)$ is one.
	Note that the virtual element space of degree~$k=1$, introduced in Ref.~\cite{ahmad13}, has been utilised here for the approximation of fluid pressure. This facilitates the computation of the $L^2$-projection onto the space of polynomials of degree up to $1$ (which are
	required in order to define the zero-order discrete bilinear form on~$Q_h(K)$).
	Next, and in order to take advantage of the features of VEM discretisations (for instance, estimation of the terms of the discrete formulation
	without explicit computation of basis functions), we need to specify the degrees of freedom associated with~\eqref{VE-spaces}. These
	entities will consist
	of discrete functionals of the type (taking as an example the space for total pressure)
	$$(D_i): Z_{h|K} \to \mathbb{R}; \qquad  Z_{h|K} \ni \phi \mapsto D_i(\phi),$$
	and we start with the degrees of freedom for the local displacement space $\bV_h(K)$:
	\begin{itemize}
		\item ($D_v1$) the values of a discrete displacement $\bv_h$ at vertices of the element;
		\item ($D_v2$) the normal displacement $\bv_h \cdot \nn^e_K$ at the mid-point of each edge $e \in \partial K$.
	\end{itemize}
	Then we precise the degrees of freedom for the local fluid pressure space $Q_h(K)$:
	\begin{itemize}
		\item  ($D_q$) the values of $q_h$ at vertices of the polygonal element.
	\end{itemize}
	And similarly, the degree of freedom for the local total pressure space $Z_h(K)$:
	\begin{itemize}
		\item ($D_z$) the value of $\phi_h$ over $K$.
	\end{itemize}
	
	It has been proven elsewhere (see e.g.\ Refs.~\cite{ahmad13,antonietti14,daveiga-b13,daveiga-e13}) that these degrees of freedom
	are unisolvent in their respective spaces.
	We also define global counterparts of the local virtual element spaces as follows:
	\begin{align*}
	\bV_h &: = \lbrace \bv_h \in \bV : \bv_h |_K \in \bV_h(K) \; \forall K \in \mathcal{T}_h \rbrace, \\
	Q_h &:= \lbrace q_h \in Q : q_h |_K \in Q_h(K) \; \forall K \in \mathcal{T}_h \rbrace,\\
	Z_h &:= \lbrace \phi_h \in Z : \phi_h |_K \in Z_h(K) \; \forall K \in \mathcal{T}_h \rbrace.
	\end{align*}
	In addition, we  denote by~$N^V$~denotes the number of degrees of freedom for~$\bV_h$,
	by~$N^Q$ the number of degrees of freedom for~$Q_h$, and
	by~$\text{dof}_r(s)$  the $r$-th degree of a given function~$s$.
	
	\subsection{Projection operators}
	Besides \eqref{eq:proj-1} we need to define other projectors. Regarding restricted quantities,
	and in particular, bilinear forms restricted locally to a single element, we will use the notation
	$\mathcal{B}^K(\cdot,\cdot) = \mathcal{B}(\cdot,\cdot)|_K$
	for a generic bilinear form~$\mathcal{B}(\cdot,\cdot)$. Then we can define
	the  energy projection $\bPi^{\beps}_K : \bV_h(K) \rightarrow [\mathbb{P}_1(K)]^2$ such that
	\begin{align*}
	& a_1^K(\bPi^{\beps}_K \bv - \bv , \boldsymbol{r}) = 0 , \quad m^K(\bPi^{\beps}_K \bv - \bv,\boldsymbol{r}) = 0 \\
	& \text{for all $\bv \in \bV_h(K)$ and $\boldsymbol{r} \in [\mathbb{P}_1(K)]^2$,}
	\end{align*}	
	where  we define
	\begin{align*}
	m^K (\bv,\boldsymbol{r}):= \frac{1}{N^v_K} \sum_{i=1}^{N^v_K} \bv(V_i) \cdot \boldsymbol{r}(V_i) \quad
	\text{for $\boldsymbol{r} \in  \ker (a_1^K(\cdot,\cdot))$.}
	\end{align*}
	Then, using the degree of freedom~$(D_v1)$, we can readily
	compute the bilinear form $m^K(\bv, \boldsymbol{r})$ for all $\boldsymbol{r} \in  \ker ( a_1^K(\cdot, \cdot))$ and~$\bv \in \bV_h(K)$.
	
	Next, for all $\bv \in \bV_h(K)$ let us consider the localised form
	\begin{align*}
	a_1^K(\bv, \boldsymbol{r}) = \int_K \beps (\bv): \beps  (\boldsymbol{r}) = -\int_K \bv \cdot \bdiv \bigl( \beps(\boldsymbol{r}) \bigr) + \int_{\partial K} \bv \cdot \bigl(\beps(\boldsymbol{r}) \nn^e_K \bigr) \, \mathrm{d}s.
	\end{align*}
	One readily sees that  $\bdiv(\beps(\boldsymbol{r})) = \cero$
	and $\beps(\boldsymbol{r})$ is constant for all $\boldsymbol{r} \in [\mathbb{P}_1(K)]^2$. Therefore
	the other term can be simply rewritten as\cite{daveiga11}
	\begin{align}\label{eq:split01} \begin{split}
	& \int_{\partial K} \bv \cdot \bigl(\beps(\boldsymbol{r}) \nn^e_K\bigr) \, \mathrm{d}s \\ &  =  \sum_{e \in \partial K}
	\left\{ \bigl(\beps(\boldsymbol{r}) \nn^e_K \cdot \boldsymbol{t}^e_K \bigr) \int_e (\bv \cdot \boldsymbol{t}^e_K) +  \bigl(\beps(\boldsymbol{r}) \nn^e_K \cdot \nn^e_K \bigr) \int_e (\bv \cdot \nn^e_K) \right\}.
	\end{split}
	\end{align}
	We can compute first term on the right-hand side of \eqref{eq:split01}
	using the degree of freedom $(D_v1)$ in conjunction with the trapezoidal rule,
	whereas for the second term it suffices to use the degrees of freedom $(D_v1)$ and $(D_v2)$
	together with a Gauss-Lobatto quadrature. Thus, the operator $\bPi^{\beps}_K$ is computable on $\bV_h(K)$.
	
	We now define the $L^2$-projection on the scalar space as $\Pi^0_K: L^2(K) \rightarrow \mathbb{P}_1(K)$ such that
	\begin{align*}
	(\Pi_{K}^0 q - q, r)_{0,K} = 0, \quad q \in L^2(K), r \in \mathbb{P}_1(K),
	\end{align*}
	and we can clearly verify that $\Pi_{K}^0 q_h =   \Pi_{K}^{\nabla} q_h, \; \forall q_h \in Q_h$.
	
	Finally, we consider the $L^2$-projection onto the piecewise constant functions, $\Pi^{0,0}_K: L^2(K) \rightarrow \mathbb{P}_0(K)$ and
	 $\,\bPi^{0,0}_K: L^2(K)^2 \rightarrow \mathbb{P}_0(K)^2$, for scalar and
    vector fields, respectively.
	We observe that the latter is fully computable on the virtual
	space $\bV_h(K)$ \cite{daveiga19}.
	
	\subsection{Discrete bilinear forms and formulations}
	For all $\bu_h, \bv_h \in \bV_h(K)$ and $p_h, q_h \in Q_h(K)$ we now
	define the  local discrete bilinear forms
	\begin{align*}
	a_1^h(\bu_h, \bv_h) |_K &:= a_1^K(\bPi^{\beps}_K \bu_h, \bPi^{\beps}_K \bv_h) + S_1^K \bigl((\bI-\bPi^{\beps}_K)\bu_h, (\bI-\bPi^{\beps}_K)\bv_h \bigr),\\
	a_2^h(p_h,q_h)|_K &:= a_2^K(\Pi^{\nabla}_K p_h, \Pi^{\nabla}_K q_h) + S_2^K \bigl((I-\Pi^{\nabla}_K)p_h, (I-\Pi^{\nabla}_K)q_h \bigr), \\
	\tilde{a}_2^h( p_h,q_h)|_K &:= \tilde{a}_2^K( \Pi^0_K p_h,\Pi^0_K q_h) + S_0^K \bigl((I-\Pi^0_K)p_h, (I-\Pi^0_K)q_h \bigr),
	\end{align*}
	where the stabilisation of the bilinear forms $S_1^K(\cdot, \cdot), S_2^K(\cdot, \cdot), S_0^K(\cdot, \cdot)$
	acting on the kernel of their respective operators $\bPi^{\beps}_K,\; \Pi^{\nabla}_K,\; \Pi^0_K$, are defined as
	\begin{align*}
	S_1^K(\bu_h,\bv_h) &:= \sigma_1^K\sum_{l=1}^{N^V} \text{dof}_l(\bu_h) \text{dof}_l(\bv_h), \quad \bu_h, \bv_h \in \text{ker}(\bPi^{\beps}_K);\\
	S_2^K(p_h,q_h) &:= \sigma_2^K\sum_{l=1}^{N^Q} \text{dof}_l(p_h) \text{dof}_l(q_h), \quad p_h, q_h \in \text{ker} (\Pi^{\nabla}_K); \\
	S_0^K(p_h,q_h) &:= \sigma_0^K\text{area}(K) \sum_{l=1}^{N^Q} \text{dof}_l(p_h) \text{dof}_l(q_h), \quad p_h, q_h \in \text{ker} (\Pi^0_K),
	\end{align*}
	where $\sigma_1^K,\sigma_2^K$ and $\sigma_0^K$ are positive multiplicative factors to take into
	account the magnitude of the physical parameters (independent of a mesh size).
	
	Note that  for all $\bv_h \in \bV_h(K),\ q_h \in Q_h(K)$,
	these stabilising terms satisfy the following relations\cite{antonietti14,daveiga11}:
	\begin{align} \label{bound:stab}
	\begin{split}
	\alpha_*a_1^K(\bv_h,\bv_h) &\le S_1^K(\bv_h,\bv_h) \le \alpha^*a_1^K(\bv_h,\bv_h),
	\\
	\zeta_* a_2^K(q_h,q_h)& \le S_2^K(q_h,q_h) \le \zeta^*a_2^K(q_h,q_h), \\
	\tilde{\zeta}_* \tilde{a}_2^K(q_h,q_h) & \le S_0^K(q_h,q_h) \le \tilde{\zeta}^*\tilde{a}_2^K(q_h,q_h),
	\end{split}
	\end{align}
	where $\alpha_*, \alpha^*, \zeta_*, \zeta^*, \tilde{\zeta}_*, \tilde{\zeta}^*$ are positive constants independent of $K$ and $h_K$.
	Now, for all $\bu_h, \bv_h \in \bV_h,\; p_h,q_h \in Q_h$, the
	global discrete bilinear forms are specified as
	\begin{align*}
	& a_1^h(\bu_h,\bv_h):= \sum_{K \in \mathcal{T}_h} a_1^h(\bu_h,\bv_h)|_K, \quad
	a_2^h(p_h,q_h):= \sum_{K \in \mathcal{T}_h} a_2^h(p_h,q_h)|_K, \\
	& \tilde{a}_2^h(p_h,q_h):= \sum_{K \in \mathcal{T}_h} \tilde{a}_2^h(p_h,q_h)|_K, \quad b_1(\bv_h, \phi_h) := \sum_{K \in \mathcal{T}_h} b_1^K (\bv_h, \phi_h)\\
	& a_3(\psi_h, \phi_h) := \sum_{K \in \mathcal{T}_h} a_3^K(\psi_h, \phi_h), \quad b_2(q_h, \phi_h) := \sum_{K \in \mathcal{T}_h} b_2^K(q_h, \phi_h)
	\end{align*}

	In addition, we observe that
	\begin{equation}\label{bfb2}
	b_2(p_h,\phi_h)= \frac{\alpha}{\lambda} \sum_{K \in \mathcal{T}_h} \int_{K} p_h\phi_h
	=\frac{\alpha}{\lambda}\sum_{K \in \mathcal{T}_h}\int_{K}\Pi^0_Kp_h\phi_h.
	\end{equation}
	
	On the other hand, the discrete linear functionals, defined on each element $K$, are
	\begin{align*}
	F^h(\bv_h)|_K:= \rho \int_K  \bb_h(\cdot, t) \cdot \bv_h, \quad \bv_h \in \bV_h; \quad
	G^h(q_h)|_K:= \int_K \ell_h(\cdot,t) q_h, \quad q_h \in Q_h,
	\end{align*}
	where the discrete load and volumetric source are given by:
	$$\bb_h(\cdot, t)|_K:= \bPi^{0,0}_K \bb (\cdot,t), \quad \ell_h(\cdot, t)|_K:= \Pi^0_K \ell(\cdot, t). $$
	In view of \eqref{bound:stab}, the discrete bilinear forms $a_1^h(\cdot, \cdot)$, $\tilde{a}_2^h(\cdot, \cdot) $ and $a_2^h(\cdot, \cdot)$ are coercive and bounded in the following manner \cite{antonietti14,daveiga-b13,vb2016}
		\begin{align*}
		a_1^h(\bu_h, \bu_h) & \ge \min \lbrace 1, \alpha_* \rbrace\,  2 \mu \, \| \beps(\bu_h) \|_0^2 & \quad \text{for all  $\bu_h \in \bV_h$,} \\
		a_2^h(q_h, q_h) & \ge \min \lbrace 1, \zeta_* \rbrace\, \frac{\kappa_{\min}}{\eta} \, \| \nabla q_h \|_0^2 & \quad \text{for all  $q_h \in Q_h$,} \\
		\tilde{a}_2^h(q_h, q_h) & \ge \min \lbrace 1, \tilde{\zeta}_* \rbrace\, \Big( c_0 +  \frac{\alpha^2}{\lambda} \Big) \, \| q_h \|_0^2 & \quad \text{for all  $q_h \in Q_h$,} \\
		a_1^h(\bu_h, \bv_h) & \le \max \lbrace 1, \alpha^* \rbrace\, 2 \mu \, \| \beps(\bu_h) \|_0 \| \beps (\bv_h) \|_0 & \quad \text{for all  $\bu_h, \bv_h \in \bV_h$,} \\
		a_2^h(p_h, q_h) & \le \max \lbrace 1, \zeta^* \rbrace\, \frac{\kappa_{\max}}{\eta} \, \| \nabla p_h \|_0 \| \nabla q_h \|_0 & \quad \text{for all  $p_h, q_h \in Q_h$,} \\
		\tilde{a}_2^h(p_h, q_h) & \le \max \lbrace 1, \tilde{\zeta}^* \rbrace\, \Big( c_0 +  \frac{\alpha^2}{\lambda} \Big)\, \| p_h \|_0 \| q_h \|_0 & \quad \text{for all  $p_h, q_h \in Q_h$}.
		\end{align*}
		Moreover, by using definitions of the operators $\bPi^{0,0}_K$ and $\Pi^0_K$, the linear functionals hold the following bounds:
		\begin{align*}
		F^h(\bv_h) & \le \rho \| \bb\|_0 \|\bv_h\|_0 & \quad \text{for all  $\bv_h \in \bV_h$,} \\
		G^h(q_h) & \le \| \ell \|_0 \| q_h \|_0 & \quad \text{for all  $q_h \in Q_h$}.
		\end{align*}

	We also recall that the bilinear form $b_1(\cdot,\cdot)$ satisfies
	the following discrete inf-sup condition on $\bV_h\times Z_h$:
	there exists $\tilde{\beta}>0$, independent of $h$, such that (see Ref.~\cite{antonietti14}),
	\begin{equation}\label{discr-infsup}
	\sup_{\bv_h (\neq 0) \in \bV_h} \frac{b_1(\bv_h, \phi_h)}{\| \bv_h \|_1}
	\ge \tilde{\beta} \| \phi_h \|_0 \quad \text{for all $\phi_h \in Z_h$.}
	\end{equation}

	The semidiscrete virtual element formulation is now defined as follows:  For all $t>0$, given $\bu_h(0)$, $p_h(0)$, $\psi_h(0)$, find $\bu_h \in \bL^2((0,t_{\text{final}}],\bV_h)$, $ p_h \in L^2((0,t_{\text{final}}],Q_h), \;\psi_h \in L^2((0,t_{\text{final}}],Z_h)$ with $ \partial_t p_h \in L^2((0,t_{\text{final}}],Q_h)$, $\partial_t \psi_h \in L^2((0,t_{\text{final}}],Z_h)$ such that
	\begin{alignat}{5}
	&&   a_1^h(\bu_h,\bv_h)   &&                 &\;+&\; b_1(\bv_h,\psi_h)     &=&\;F^h(\bv_h)&\quad\forall \bv_h \in \bV_h, \label{weak-uh}\\
	\tilde{a}_2^h(\partial_t p_h,q_h) &\; +&               &&      a_2^h(p_h,q_h)   &\;-&\;   b_2( q_h, \partial_t \psi_h)  &=&\;G^h(q_h) &\quad\forall q_h \in Q_h, \label{weak-ph}\\
	&&b_1(\bu_h,\phi_h)  &\;+\;& b_2(p_h,\phi_h)&\;-&\; a_3(\psi_h,\phi_h) &=&0 &\quad\forall\phi_h \in Z_h. \label{weak-psih}
	\end{alignat}
	
	Now we establish the stability of \eqref{weak-uh}--\eqref{weak-psih}.
	\begin{theorem}[Stability of the semi-discrete problem]
		Let $(\bu_h(t), p_h(t), \psi_h(t))$ be a solution of problem
		\eqref{weak-uh}--\eqref{weak-psih} for each $t \in (0,t_{\text{final}}]$. Then  there exists a constant $C$ independent of $h, \lambda$ such that
		\begin{align} \label{bound:semi-stability} \begin{split}
		& \mu \| \beps(\bu_h(t)) \|_0^2  + \| \psi_h(t) \|_0^2 + c_0 \| p_h(t)\|_0^2 + \frac{\kappa_{\min}}{\eta} \int_0^t\| \nabla p_h(s) \|_0^2 \, \mathrm{d}s  \\
		&   \le C \bigg( \| \beps(\bu_h(0))\|_0^2 + \| p_h(0)\|_0^2 + \| \psi_h(0) \|_0^2 +
		\int_0^t \| \partial_t \bb(s)\|_0^2 \, \mathrm{d}s \\
		& \qquad \qquad + \sup_{t \in [0,t_{\text{final}}] }  \| \bb(t)\|_0^2 + \int_0^t \| \ell(s) \|_0^2 \, \mathrm{d}s \bigg).
		\end{split} \end{align}		
	\end{theorem}	
	\begin{proof}
		Following Ref. \cite{lee19}, we can differentiate equation \eqref{weak-psih} with respect to time and choose
		as test function $\phi_h= -\psi_h$. We get
		\begin{align*}
		-b_1(\partial_t \bu_h, \psi_h) - b_2(\partial_t p_h, \psi_h) + a_3(\partial_t \psi_h, \psi_h) = 0.
		\end{align*}
		Then we take $q_h = p_h$ in \eqref{weak-ph}, $\bv_h = \partial_t \bu_h$ in \eqref{weak-uh} and add the result to the previous relation to obtain
		\begin{align*}
		& a_1^h(\bu_h, \partial_t \bu_h) + b_1(\partial_t \bu_h, \psi_h) + \tilde{a}_2^h(\partial_t p_h, p_h) + a_2^h(p_h, p_h) - b_2(p_h, \partial_t \psi_h)\\ &  - b_1(\partial_t \bu_h, \psi_h)
		- b_2(\partial_t p_h, \psi_h) + a_3(\partial_t \psi_h, \psi_h) = F^h(\partial_t \bu_h) + G^h(p_h).
		\end{align*}
		Using the stability of the bilinear forms $a_1^h(\cdot,\cdot)$, $a_2^h(\cdot, \cdot)$, $ \tilde{a}_2^h(\cdot, \cdot)$ as well as the definition
		of the discrete bilinear forms $b_1(\cdot, \cdot)$ (cf. \eqref{bfb2}) and
		$ \tilde{a}_2^h(\cdot, \cdot)$, we readily have
		\begin{align} \label{semi:stab}
		& \frac{\mu}{2} \frac{\mathrm{d}}{\mathrm{d}t} \| \beps (\bu_h) \|_0^2 + \frac{c_0}{2} \frac{\mathrm{d}}{\mathrm{d}t} \| p_h \|_0^2 + \frac{\kappa_{\min}}{\eta} \| \nabla p_h\|_0^2 + \frac{1}{\lambda} \|\psi_h\|_{0,K}^2 \nonumber \\
		& + \sum_K \bigg(  \frac{\alpha^2}{\lambda}  \Big(  \bigl(\partial_t(\Pi^0_K p_h), \Pi^0_K p_h \bigr)_{0,K} + S^K_0
		\bigl((I-\Pi^0_K)\partial_t p_h,  (I-\Pi^0_K) p_h \bigr) \Big) \nonumber\\
		& \qquad \qquad - \frac{\alpha}{\lambda} \Big(  (\Pi^0_K p_h, \partial_t \psi_h)_{0,K} +
		\bigl(\partial_t (\Pi^0_K p_h), \psi_h \bigr)_{0,K} \Big) \biggr)
		\\
		& \lesssim F^h(\partial_t \bu_h) + G^h(p_h). \nonumber
		\end{align}
		Rearranging terms  on the left-hand side gives
		\begin{align*}
		& \frac{\mu}{2} \frac{\mathrm{d}}{\mathrm{d}t} \| \beps (\bu_h) \|_0^2 + \frac{\kappa_{\min}}{\eta} \| \nabla p_h\|_0^2 + \frac{c_0}{2} \frac{\mathrm{d}}{\mathrm{d}t} \| p_h \|_0^2 \\
		& + \frac{1}{\lambda} \sum_K \bigg( \bigl(\partial_t(\alpha \Pi^0_K p_h - \psi_h), (\alpha \Pi^0_K p_h - \psi_h)\bigr)_{0,K}  \\
		& \qquad \qquad + \frac{\alpha^2}{2} \frac{\mathrm{d}}{\mathrm{d}t} S^K_0
		\bigl((I-\Pi^0_K) p_h, (I-\Pi^0_K) p_h \bigr) \bigg) \lesssim F^h(\partial_t \bu_h) + G^h(p_h),
		\end{align*}
		and after exploiting the stability of $S^K_0(\cdot, \cdot)$ and integrating from $0$ to $t$,
		we  arrive at
		\begin{align*}
		& \mu \| \beps (\bu_h (t))\|_0^2 + c_0 \| p_h(t)\|_0^2 +  \frac{\alpha^2}{\lambda} \sum_{K} \| (I-\Pi^0_K) p_h(t)\|_{0,K}^2 \\
		& \qquad \qquad + \frac{1}{\lambda} \sum_{K} \| (\alpha \Pi^0_K p_h - \psi_h)(t)\|_{0,K}^2  + \frac{\kappa_{\min}}{\eta} \int_0^t \| \nabla p_h(s) \|_0^2 \, \mathrm{d}s\\
		& \lesssim \mu \| \beps (\bu_h (0))\|_0^2 + c_0 \| p_h(0)\|_0^2 +  \frac{\alpha^2}{\lambda} \sum_{K} \| (I-\Pi^0_K) p_h(0)\|_{0,K}^2  \\
		& \qquad \qquad + \frac{1}{\lambda} \sum_{K} \| (\alpha \Pi^0_K p_h - \psi_h)(0)\|_{0,K}^2  \\
		& \qquad \qquad +  \underbrace{ \rho \int_0^t  \sum_K  \bigl(\bb(s), \bPi^{0,0}_K \partial_t \bu_h(s) \bigr)_{0,K}}_{=:T_1}
		+ \underbrace{\int_0^t  \sum_K  \bigl( \ell(s), \Pi^0_K p_h(s) \bigr)_{0,K} }_{=:T_2}.
		\end{align*}
		Then, integration by parts in time, and an application of  Korn, Poincar\'e, and Young inequalities, implies that
		\begin{align*}
		T_1 & =  \rho \sum_K \Bigl(  \bigl(\bb(t), \bPi^{0,0}_K  \bu_h(t) \bigr)_{0,K} - \bigl(\bb(0), \bPi^{0,0}   _K  \bu_h(0) \bigr)_{0,K} \Bigr) \\
		& \quad - \rho \int_0^t \sum_K  \bigl(\partial_t \bb(s), \bPi^{0,0}_K  \bu_h(s) \bigr)_{0,K} \, \mathrm{d}s \\
		&  \le \mu \|  \beps(\bu_h(t))\|_0^2 \\ & \quad +  C_1  \rho \left(\frac{\rho}{\mu} \| \bb(t)\|_0^2 +  \| \bb(0)\|_0 \| \beps(\bu_h(0))\|_0 + \int_0^t \| \partial_t \bb (s) \|_0 \| \beps(\bu_h(s))\|_0 \, \mathrm{d}s \right).
		\end{align*}
		The bound for  $T_2$ follows from the Cauchy-Schwarz, Poincar\'e, and Young inequalities in the following manner:
		\begin{align*}
		T_2 & = \int_0^t  \sum_K ( \ell (s), \Pi^0_K p_h(s))_{0,K} \, \mathrm{d}s  \\
		& \lesssim \int_0^t \| \ell(s)\|_0 \| p_h(s) \|_0 \, \mathrm{d}s   \le  C_2\frac{\eta}{\kappa_{\min}} \int_0^t \| \ell(s)\|_0^2 \, \mathrm{d}s   + \frac{\kappa_{\min}}{2\eta} \int_0^t \| \nabla p_h(s) \|_0^2 \, \mathrm{d}s  .
		\end{align*}
		Thus, we achieve
		\begin{align} \label{bound:semi_u-p}
		\begin{split}
		& \mu \| \beps (\bu_h (t))\|_0^2 + c_0 \| p_h(t)\|_0^2 + \frac{\alpha^2}{\lambda}\sum_{K} \| (I-\Pi^0_K) p_h(t)\|_{0,K}^2 \\
		& \qquad \qquad + \frac{1}{\lambda} \sum_{K} \| (\alpha \Pi^0_K p_h - \psi_h)(t)\|_{0,K}^2  + \frac{\kappa_{\min}}{2\eta} \int_0^t \| \nabla p_h(s) \|_0^2 \, \mathrm{d}s   \\
		& \lesssim \mu \| \beps (\bu_h (0))\|_0^2 + c_0 \| p_h(0)\|_0^2 + \frac{\alpha^2}{\lambda}\sum_{K} \| (I-\Pi^0_K) p_h(0)\|_{0,K}^2  \\
		& \qquad \qquad + \frac{1}{\lambda} \sum_{K} \| (\alpha \Pi^0_K p_h - \psi_h)(0)\|_{0,K}^2 + C \bigg( \int_0^t \| \ell(s)\|_0^2 \, \mathrm{d}s + \Big(\| \bb(t)\|_0^2  \\
		& \qquad \qquad \qquad     + \| \bb(0)\|_0 \| \beps(\bu_h(0))\|_0 + \int_0^t \| \partial_t \bb (s) \|_0 \| \beps(\bu_h(s))\|_0 \, \mathrm{d}s \Big)  \bigg).
		\end{split}
		\end{align}
		The discrete inf-sup condition \eqref{discr-infsup}
		alongwith \eqref{weak-uh} gives
		\begin{align} \label{bound:semi_psi}
		\| \psi_h \|_0 \le \sup_{\bv_h (\neq 0) \in \bV_h} \frac{1}{\| \bv_h \|_1} \big( F^h(\bv_h) - a_1^h(\bu_h, \bv_h) \big) \le C (\| \bb\|_0 + \| \beps(\bu_h)\|_0 ).
		\end{align}
	    Now, Young's and Gronwall's inequalities together with
	    \eqref{bound:semi_u-p}-\eqref{bound:semi_psi} concludes
	    the proof of the bound \eqref{bound:semi-stability}.
	\end{proof}
	
	\begin{corollary}[Solvability of the discrete problem]
		The problem \eqref{weak-uh}-\eqref{weak-psih} has a unique solution in $\bV_h \times Q_h \times Z_h$ for each $t \in (0,t_{\mathrm{final}}]$.
	\end{corollary}
	\begin{proof}
		Analogously to the Fredholm alternative approach exploited in Ref.~\cite{oyarzua16}, one can consider \eqref{weak-uh}--\eqref{weak-psih} as the operator problem of finding $\vec{\bu}_h(t):= (\bu_h(t), p_h(t), \psi_h(t))$ such that
		$$
		(\mathcal{A}^h + \mathcal{B}^h) \vec{\bu}_h(t) = \mathcal{F}^h,
		$$
		where
		\begin{align*}
		\langle \mathcal{A}^h (\vec{\bu}_h), \vec{\bv}_h\rangle & := a_1^h(\bu_h, \bv_h) + \tilde{a}_2^h(\partial_t p_h, q_h) + a_2^h(p_h, q_h) + a_3(\psi_h, \phi_h) \\& \quad + b_1(\bv_h, \psi_h)- b_1(\bu_h, \phi_h), \\
		\langle \mathcal{B}^h (\vec{\bu}_h), \vec{\bv}_h\rangle &:= -b_2(q_h, \partial_t \psi_h) - b_2(p_h, \phi_h).
		\end{align*}
		Note that one can regard
		the problem $\mathcal{A}^h \vec{\bu}_h = \mathcal{L}^h$ for given $\mathcal{L}^h=(L_1^h,L_2^h,L_3^h) \in (\bV_h \times Q_h \times Z_h)'$ as a combination of the perturbed saddle-point problem
		\begin{align*}
		& \text{For every $t \in (0, t_{\text{final}}]$, find $(\bu_h(t),\psi_h(t)) \in \bV_h \times Z_h$ such that} \\
		& a_1^h(\bu_h, \bv_h) + b_1(\bv_h, \psi_h) = L_1^h(\bv_h) \quad \text{for all $\bv_h \in \bV_h$,}  \\
		& b_1(\bu_h, \phi_h) - a_3(\psi_h, \phi_h) = L_3^h(\phi_h) \quad \text{for all $\phi_h \in Z_h$}
		\end{align*}
		and the parabolic problem
		\begin{align*}
		& \text{For each $t \in (0, t_{\text{final}}]$, find $p_h(t) \in Q_h$ such that} \\
		& \tilde{a}_2^h(\partial_t p_h, q_h) + a_2^h(p_h, q_h) = L_2^h(q_h) \quad \text{for all $q_h \in Q_h$.}
		\end{align*}
		Classical saddle-point theory\cite{boffi13} and the theory of parabolic problems\cite{ladyzenskaja68} then imply  the invertibility of the operator
		$\mathcal{A}^h$. On the other hand,
		noting that $\partial_t \psi_h \in L^2((0,t_{\text{final}}],Z_h)$ and that the operator induced by $b_2(\cdot,\cdot)$ from $\bV_h$ to $Z_h$ is compact (and so is its adjoint), we obtain that  the operator~$\mathcal{B}^h$ is compact for a given $t \in (0, t_{\text{final}}]$. Hence the unique solvability is obtained by invoking the stability result~\eqref{bound:semi-stability}.
	\end{proof}

	\noindent Next, we discretise in time using the backward Euler method with the constant step size $\Delta t = t_{\mathrm{final}} / N$ and denote any function $f$ at $t=t_n$ by $f^n$. The fully discrete scheme reads:
	\begin{subequations}  \label{fd-scheme}
		\begin{align}
		& \text{Given $\bu_h^0$, $p_h^0$, $\psi_h^0$, and for $t_n=n\Delta t$, $n=1, \dots, N$, find $\bu_h^{n} \in \bV_h$,}  \nonumber \\
		& \text{$p_h^{n} \in Q_h$ and $\psi_h^{n} \in Z_h$ such that for all $\bv_h \in \bV_h$, $q_h \in Q_h$ and $\phi_h \in Z_h$} \nonumber  \\
		&	a_1^h(\bu_h^{n},\bv_h)  + b_1(\bv_h,\psi_h^{n})     = F^{h,n}(\bv_h), \label{weak-uhn}\\
		&	\tilde{a}_2^h \left( p_h^{n} , q_h \right) + \Delta t a_2^h( p_h^{n},q_h) - b_2 \left( q_h, \psi_h^{n} \right)  \nonumber
		\\ &  = \Delta t G^{h,n} (q_h)+\tilde{a}_2^h \left(  p_h^{n-1} , q_h \right) -b_2 \left( q_h, \psi_h^{n-1} \right), \label{weak-phn}\\
		&	b_1(\bu_h^{n},\phi_h)  + b_2(p_h^{n},\phi_h) - a_3(\psi_h^{n},\phi_h)  = 0, \label{weak-psihn}
		\end{align}
	\end{subequations} 	
	where 	for all $\bv_h \in \bV_h$ and $q_h \in Q_h$ we define
	\begin{align*}
	F^{h,n}(\bv_h)|_K:= \rho\int_K  \bb_h(t^n) \cdot \bv_h, \quad
	G^{h,n}(q_h)|_K:= \int_K \ell_h(t^n) q_h.
	\end{align*}
	
	\begin{theorem}[Stability of the fully-discrete problem]
		The unique solution to problem \eqref{fd-scheme} depends   continuously  on data.
		Precisely, there exists a constant~$C$ independent of $\lambda, h, \Delta t$ such that
		\begin{align} \label{stab-ineq}  \begin{split}
		& \mu \| \beps (\bu_h^n) \|_0^2 + \| \psi_h^n \|_0^2 + c_0 \| p_h^n \|_0^2 + (\Delta t)\frac{\kappa_{\text{min}}}{\eta} \sum_{j=1}^n \| \nabla p_h^j \|_0^2 \\
		& \le C \bigg( \| \beps (\bu_h^0)\|_0^2 + \| p_h^0 \|_0^2 + \|\psi_h^0 \|_0^2 + \max_{0 \le j \le n} \| \bb^j \|_0^2 \\
		& \qquad \quad  + (\Delta t) \sum_{j=1}^n \Big( \| \partial_{t} \bb^j \|_0^2+ \|\ell^j\|_0^2 \Big) + (\Delta t)^2 \int_{0}^{T} \| \partial_{tt} \bb(s) \|_0^2 \, \mathrm{d}s \bigg).
		\end{split}
		\end{align}
		with $\bb^k:=\bb(\cdot,t^k)$ and $\ell^k:=\ell(\cdot,t^k)$, for $k=1,\ldots,n$.
	\end{theorem}
	\begin{proof}
		Taking $\bv_h = \bu_h^n - \bu_h^{n-1}$ in \eqref{weak-uhn} gives
		\begin{align}\label{eq:un}
		a_1^h(\bu_h^n, \bu_h^n - \bu_h^{n-1}) +  b_1(\bu_h^n - \bu_h^{n-1}, \psi_h^n)= F^{h,n}(\bu_h^n - \bu_h^{n-1}).
		\end{align}
    A use of \eqref{weak-psih} for the time step $n$, $n-1$ and setting $\phi_h = -\psi_h^n$,  \eqref{weak-psihn} becomes
		\begin{align} \label{eq:psihn}
		-b_1(\bu_h^n - \bu_h^{n-1}, \psi_h^n) - b_2(p_h^n - p_h^{n-1}, \psi_h^n) + a_3(\psi_h^n - \psi_h^{n-1}, \psi_h^n) = 0.
		\end{align}
		Adding \eqref{eq:psihn} from \eqref{eq:un} we readily obtain
		\begin{align} \label{eq:u-psi} \begin{split}
		a_1^h(\bu_h^n, \bu_h^n-\bu_h^{n-1}) + a_3(\psi_h^n - \psi_h^{n-1}, \psi_h^n)-b_2(p_h^n - p_h^{n-1}, \psi_h^n) \\
		= F^{h,n}(\bu_h^n - \bu_h^{n-1}),
		\end{split}
		\end{align}
		and choosing $q_h = p_h^n$ in \eqref{weak-phn} implies the relation
		\begin{align}\label{eq:pn} \begin{split}
		\tilde{a}_2^h(p_h^n - p_h^{n-1},  p_h^n) + \Delta t \, a_2^h(p_h^n,  p_h^n) - b_2( p_h^n, \psi_h^n - \psi_h^{n-1}) = \Delta t \, G^{h,n}(p_h^n).
		\end{split}
		\end{align}
		Next we proceed to adding \eqref{eq:u-psi} and \eqref{eq:pn}, to get
		\begin{align} \label{eq:combined} \begin{split}
		& 	a_1^h(\bu_h^n, \bu_h^n-\bu_h^{n-1})   + \Delta t \,  a_2^h(p_h^n,  p_h^n) + a_3(\psi_h^n - \psi_h^{n-1}, \psi_h^n) \\
		&  \qquad \qquad + \tilde{a}_2^h(p_h^n - p_h^{n-1},  p_h^n) -b_2(p_h^n - p_h^{n-1}, \psi_h^n) - b_2( p_h^n, \psi_h^n - \psi_h^{n-1}) \\
		&  = F^{h,n}(\bu_h^n-\bu_h^{n-1}) + \Delta t \,  G^{h,n}(p_h^n). \end{split}
		\end{align}
		Repeating the similar argument (as to obtain \eqref{semi:stab}) used in the derivation of proof of stability of semi-discrete scheme together with the inequality
		\begin{align} \label{bound:discrete_ineq}
		(f_h^n - f_h^{n-1}, f_h^n) \ge \frac{1}{2} (\| f_h^n \|_0^2 - \| f_h^{n-1}\|_0^2),
		\end{align}
		for any discrete function $f_h^j,\, j=1, \dots, n$ we arrive at
		\begin{align*}
		& \frac{\mu}{2} (\| \beps (\bu_h^n) \|_0^2 - \| \beps (\bu_h^{n-1})\|_0^2) + (\Delta t)\frac{\kappa_{\text{min}}}{\eta} \| \nabla p_h^n \|_0^2 \\
		& \qquad + \frac{1}{2} \sum_{K} c_0 (\| \Pi^0_K p_h^n \|_{0,K}^2 - \| \Pi^0_K p_h^{n-1} \|_{0,K}^2) \\
		& \qquad + \frac{1}{2} \Big(c_0 + \frac{\alpha^2}{\lambda} \Big) \sum_{K}  (\|  (I - \Pi^0_K) p_h^n \|_{0,K}^2 - \|  (I - \Pi^0_K) p_h^{n-1} \|_{0,K}^2) \\
		& \qquad + \frac{1}{2 \lambda} \sum_{K} ( \| \alpha \Pi^0_K p_h^n - \psi_h^n \|_{0,K}^2 - \| \alpha \Pi^0_K p_h^{n-1} - \psi_h^{n-1} \|_{0,K}^2) \\
		& \lesssim (\Delta t) ( \rho (\bb_h^n, \delta_t \bu_h^n)_{0,\Omega} +  (\ell_h^n, p_h^n)_{0,\Omega}).
		\end{align*}
		where we have denoted $\delta_t f_h(t_n): = \frac{f_h(t_{n}) - f_h(t_{n-1})}{\Delta t}$ for any time-space discrete function $f_h$. Summing over $n$ we obtain
		\begin{align*}
		& \frac{\mu}{2} (\| \beps (\bu_h^n) \|_0^2 - \| \beps (\bu_h^0)\|_0^2) + (\Delta t)\frac{\kappa_{\text{min}}}{\eta} \sum_{j=1}^n \| \nabla p_h^j \|_0^2 \\
		& \qquad + \frac{1}{2} \sum_{K} c_0 (\| \Pi^0_K p_h^n \|_{0,K}^2 - \| \Pi^0_K p_h^0 \|_{0,K}^2) \\
		& \qquad + \frac{1}{2} \Big(c_0 + \frac{\alpha^2}{\lambda} \Big) \sum_{K}  (\|  (I - \Pi^0_K) p_h^n \|_{0,K}^2 - \|  (I - \Pi^0_K) p_h^0 \|_{0,K}^2) \\
		& \qquad + \frac{1}{2 \lambda} \sum_{K} ( \| \alpha \Pi^0_K p_h^n - \psi_h^n \|_{0,K}^2 - \| \alpha \Pi^0_K p_h^0- \psi_h^0 \|_{0,K}^2) \\
		& \lesssim  \underbrace{\rho (\Delta t) \sum_{j=1}^n (\bb_h^j, \delta_t \bu_h^j)_{0,\Omega}}_{=:J_1} +  \underbrace{(\Delta t) \sum_{j=1}^n (\ell_h^j, p_h^j)_{0,\Omega}}_{=:J_2}.
		\end{align*}
		Using the equality
		\begin{align}	\label{eq:discrete_equ}
		\sum_{j=1}^n (f_h^j - f_h^{j-1}, g_h^j )= (f_h^n, g_h^n) - (f_h^0, g_h^0)-\sum_{j=1}^n (f_h^{j-1}, g_h^j- g_h^{j-1}),
		\end{align} for any discrete functions $f_h^j, g_h^j$, $j=1, \dots, n$, alongwith the Taylor expansion, Cauchy Schwarz, Korn's inequality and generalised Young's inequality gives
		\begin{align*}
		J_1 & = \rho \Big( (\bb_h^n, \bu_h^n)_{0,\Omega} - (\bb_h^0, \bu_h^0)_{0,\Omega}-\sum_{j=1}^n (\bb_h^j - \bb_h^{j-1}, \bu_h^{j-1})_{0,\Omega} \Big) \\
		& = \rho \Big( (\bb_h^n, \bu_h^n)_{0,\Omega} - (\bb_h^0, \bu_h^0)_{0,\Omega}- (\Delta t) \sum_{j=1}^n ( \partial_{t} \bb_h^j, \bu_h^{j-1})_{0,\Omega} \\
		& \qquad \quad + \sum_{j=1}^n \Big( \int_{t_{j-1}}^{t_j} (s-t_{j-1}) \partial_{tt} \bb_h(s)\, \mathrm{d}s, \bu_h^{j-1} \Big)_{0,\Omega} \Big)
		\\
		& \le \mu \| \beps(\bu_h^0) \|_0^2 + \frac{\mu}{4} \| \beps(\bu_h^n) \|_0^2 + \mu (\Delta t) \sum_{j=0}^{n-1} \| \beps (\bu_h^{j}) \|_0^2 \\
		& \qquad \quad + C_1(\rho, \mu) \Big( \max_{0 \le j \le n} \| \bb^j \|_0^2 + (\Delta t) \sum_{j=1}^n \|  \partial_{t} \bb^j \|_0^2 + (\Delta t)^2 \int_{0}^{T} \| \partial_{tt} \bb(s) \|_0^2 \, \mathrm{d}s \Big).
		\end{align*}
		Again an application of Young's inequality gives
		\begin{align*}
		J_2 \le C_2(\eta, \kappa_{\min}) (\Delta t) \sum_{j=1}^n \| \ell^j \|_0^2 + (\Delta t) \frac{\kappa_{\min}}{2 \eta } \sum_{j=1}^n \| p_h^j \|_0^2.
		\end{align*}
		Bounds of $J_1$,  $J_2$ and $\Pi^0_K$ implies
		\begin{align} \label{bound:discete_u-p}
		& \mu \| \beps (\bu_h^n) \|_0^2 + c_0 \| p_h^n \|_0^2 + (\Delta t)\frac{\kappa_{\text{min}}}{\eta} \sum_{j=1}^n \| \nabla p_h^j \|_0^2 +  \Big(\frac{\alpha^2}{\lambda} \Big) \sum_{K} \|  (I - \Pi^0_K) p_h^n \|_{0,K}^2 \nonumber \\
		& \qquad + \frac{1}{\lambda} \sum_{K} \| \alpha \Pi^0_K p_h^n - \psi_h^n \|_{0,K}^2 \nonumber\\
		& \le \frac{\mu}{2} \| \beps( \bu_h^n) \|_0^2 + (\Delta t) \Big( \frac{\kappa_{\min}}{2 \eta } \sum_{j=1}^n \| p_h^j \|_0^2 + \mu \sum_{j=0}^{n-1} \| \beps(\bu_h^j) \|_0^2 \Big)\\
		& \qquad + C \Big(\| \beps (\bu_h^0)\|_0^2 + \| p_h^0 \|_0^2 + \|\psi_h^0 \|_0^2 \nonumber \\
		& \qquad \qquad+ \max_{0 \le j \le n} \| \bb^j \|_0^2 + (\Delta t) \sum_{j=1}^n \| \partial_{t} \bb^j \|_0^2 + (\Delta t)^2 \int_{0}^{T} \| \partial_{tt} \bb(s) \|_0^2 \, \mathrm{d}s \nonumber \\
		& \qquad \qquad + (\Delta t) \sum_{j=1}^n \|\ell^j\|_0^2 \Big). \nonumber
		\end{align}
		An application of \eqref{discr-infsup} together with \eqref{weak-uhn} yields
		\begin{align} \label{bound:discrete_psi}
		\|\psi_h^n \|_0 \le C (\|\bb^n\|_0 + \| \beps(\bu_h^n)\|_0).
		\end{align}
		Finally, the discrete Gronwall's inequality and \eqref{bound:discete_u-p}-\eqref{bound:discrete_psi} concludes \eqref{stab-ineq}.
	\end{proof}

	It is worth  pointing  out that the proof is particularly delicate since the stabilisation term requires
	a careful treatment in order to
	guarantee that the bounds remain independent of the stability constants of the bilinear form
	$\tilde{a}_2(\cdot, \cdot)$.

	\section{A priori error estimates} \label{sec:estimates}
		For the sake of error analysis, we require the high
	regularity: In particular, for any $t>0$, we consider that the displacement is $\bu(t) \in H^{2}(\Omega)$,
	the fluid pressure $ p(t)\in H^{2}(\Omega)$, and the total pressure
	$\psi(t) \in H^{1}(\Omega)$. We recall the
	estimate for the interpolant  $\bu_I \in \bV_h$ of $\bu$ and
	$p_I \in Q_h$ of $p$ (see Refs. \cite{antonietti14,CGPS,CMS2016,MRR2015}).
	\begin{lemma} \label{interpolant_u}
		There exist interpolants $\bu_I \in \bV_h$  and  $p_I \in Q_h$     of~$\bu$ and~$p$, respectively,
		such that
		\begin{align*}
		\| \bu - \bu_I \|_0 + h | \bu - \bu_I|_1 \leq C h^{2} |\bu|_2, \quad
		\| p - p_I \|_0 + h | p - p_I|_1 \leq C h^{2} |p|_2.
		\end{align*}
	\end{lemma}

We now introduce the poroelastic projection operator: given
	$(\bu,p,\psi)\in \bV\times Q \times Z$, find $I^h :=( I^h_{\bu} \bu, I^h_p p, I^h_{\psi} \psi)$
	$\in \bV_h \times Q_h \times Z_h$ such that
	\begin{alignat}{5}
	a_1^h(I_{\bu}^h \bu,\bv_h) &\;+&\; b_1(\bv_h, I_{\psi}^h \psi) & = &a_1(\bu,\bv_h)  &\;+&\; b_1(\bv_h, \psi) &\quad\text{for all $\bv_h \in \bV_h$,}  \label{weak-Iuh}\\
	& \;&\; b_1(I_{\bu}^h \bu, \phi_h)  &=& b_1(\bu, \phi_h)   & \;&\;  &\quad \text{for all $\phi_h \in Z_h$,}  \label{weak-Ipsih} \\
	& \;&\; a_2^h(I^h_p p,q_h)  &=& a_2(p,q_h)   & \;&\;  &\quad \text{for all $q_h \in Q_h$,}  \label{weak-Iph}
	\end{alignat}
	and we remark that $I^h$ is defined by the combination of the
	saddle-point problem \eqref{weak-Iuh}, \eqref{weak-Ipsih} and the
	elliptic problem \eqref{weak-Iph}; and hence, it is well-defined.
	\begin{theorem}[Estimates for the poroelastic  projection]
		Let $(\bu, p, \psi)$ and $( I^h_{\bu} \bu, I^h_p p, I^h_{\psi} \psi)$ be the unique
		solutions of \eqref{weak-uh}--\eqref{weak-psih} and \eqref{weak-Iuh}, \eqref{weak-Ipsih},
		respectively. Then the following estimates hold:
		\begin{align}
		\| \bu - I^h_{\bu} \bu \|_0 + h \| \bu - I^h_{\bu} \bu \|_1 &\le  C h^2 (|\bu|_2 + |\psi|_1), \label{estimate-Ihu}\\
		\| \psi - I^h_{\psi} \psi \|_0  &\le  C h (|\bu|_2 + |\psi|_1), \label{estimate-Ihpsi}\\
		\| p - I^h_p p \|_0 + h \| p - I^h_p p \|_1 &\le C h^2 |p|_2  \label{estimate-Ihp}.
		\end{align}
	\end{theorem}
	\begin{proof}
		The estimates available for discretisations of Stokes\cite{antonietti14} and elliptic problems\cite{daveiga-g13} conclude the statement.
	\end{proof}
	
	\begin{remark}
		Note that repeating the same arguments exploited in this and in the subsequent sections,
		it is possible to derive error estimates of order $h^s$. It suffices to assume that
		$\bu(t) \in H^{1+s}(\Omega)^2$,  $ p(t)\in H^{1+s}(\Omega)$, and
		$\psi(t) \in H^{s}(\Omega)$, for $0<s\le1$.
	\end{remark}
	
	\begin{theorem}[Semi-discrete energy error estimates]\label{th:semid}
		Let $(\bu(t),p(t),\psi(t)) \in \bV \times Q \times Z$
		and $(\bu_h(t),p_h(t),\psi_h(t)) \in \bV_h \times Q_h \times Z_h$ be the unique
		solutions to problems \eqref{weak-u}--\eqref{weak-psi} and \eqref{weak-uh}--\eqref{weak-psih},
		respectively. Then, the following bounds hold, with constants $C>0$ independent of~$h$, $\lambda$
		\begin{align*}
			& \mu \| \beps((\bu- \bu_h)(t)) \|_0^2  + \|(\psi - \psi_h)(t) \|_0^2  \\
			& \qquad \quad + \frac{\kappa_{\min}}{\eta} \int_0^t \| \nabla (p - p_h)(s) \|_0^2 \, \mathrm{d}s \,\le \, C\, h^2.
			\end{align*}
	\end{theorem}
	\begin{proof}
		Invoking the Scott-Dupont Theory (see Ref. \cite{brenner08})
		for the polynomial approximation: there exists a constant $C>0$
		such that for every $s$ with $0\le s\le 1$ and for every $u\in H^{1+s}(K)$,
		there exists $u_\pi\in\mathbb{P}_k(K)$, $k=0,1$, such that
		\begin{equation}\label{poly_est_u}
		\| u - u_{\pi} \|_{0,K} + h_K |u -u_{\pi} |_{1,K} \leq C h_K^{1+s} |u|_{1+s,K}
		\quad \text{for all $K \in \mathcal{T}_h$.}
		\end{equation}
		We can then write the displacement and total pressure error
		in terms of the poro\-elastic projector as follows
		\begin{align*}
		(\bu - \bu_h)(t) = (\bu - I^h_{\bu}\bu)(t) + (I^h_{\bu} \bu - \bu_h)(t) := e_{\bu}^I(t) + e_{\bu}^A(t), \\
		(\psi - \psi_h)(t) = (\psi - I^h_{\psi} \psi)(t) + (I^h_{\psi} \psi - \psi_h)(t) := e_{\psi}^I(t) + e_{\psi}^A(t).
		\end{align*}
		Then, a combination of equations \eqref{weak-Iuh}, \eqref{weak-uh} and \eqref{weak-u} gives
		\begin{align*}
		a_1^h(e_{\bu}^A, \bv_h)+ b_1(\bv_h, e_{\psi}^A) & = (a_1(\bu,\bv_h) -a_1^h(\bu_h, \bv_h)) + b_1(\bv_h, \psi- \psi_h) \\&
		= (F-F^h)(\bv_h),
		\end{align*}
		and taking as test function $\bv_h = \partial_t e_{\bu}^A$, we
		can write the relation
		\begin{align}
		a_1^h(e_{\bu}^A, \partial_t e_{\bu}^A)
		+ b_1(\partial_t e_{\bu}^A, e_{\psi}^A) = (F-F^h)(\partial_t e_{\bu}^A). \label{est-eq1}
		\end{align}
		Now, we write the pressure error in terms of the poroelastic projector as follows
		\begin{align*}
		(p - p_h)(t) = (p - I^h_p p)(t) + (I^h_p p - p_h)(t) := e_{p}^I(t) + e_{p}^A(t).
		\end{align*}
		Using \eqref{weak-Iph}, \eqref{weak-ph} and \eqref{weak-p}, we obtain
		\begin{align*}
		& \tilde{a}_2^h(\partial_t e_p^A, q_h) + a_2^h(e_p^A, q_h) - b_2(q_h, \partial_t e_{\psi}^A ) \\ & = \tilde{a}_2^h(\partial_t I^h_p p, q_h) + a_2(p, q_h ) - b_2(q_h, \partial_t I^h_{\psi} \psi) - G^h(q_h) \\
		& =  (\tilde{a}_2^h(\partial_t I^h_p p, q_h) - \tilde{a}_2(\partial_t p, q_h)) + b_2(q_h, \partial_t e_{\psi}^I)  + (G- G^h)(q_h).
		\end{align*}
		We can take $q_h = e_p^A$, which leads to
		\begin{align} \label{est-eq2} \begin{split}
		& \tilde{a}_2^h(\partial_t e_p^A, e_p^A) +a_2^h(e_p^A, e_p^A) - b_2(e_p^A, \partial_t e_{\psi}^A) \\ & = (\tilde{a}_2^h(\partial_t I^h_p p, e_p^A) - \tilde{a}_2(\partial_t p, e_p^A) ) + b_2(e_p^A, \partial_t e_{\psi}^I) + (G- G^h)(e_p^A). \end{split}
		\end{align}
		Next we use \eqref{weak-Ipsih}, \eqref{weak-psih} and \eqref{weak-psi}, and this implies
		\begin{align*}
		& b_1(e_{\bu}^A, \phi_h)+ b_2(e_p^A , \phi_h) - a_3(e_{\psi}^A, \phi_h) = b_1(I^h_{\bu} \bu, \phi_h) + b_2(I^h_p p, \phi_h) - a_3(I^h_{\psi} \psi, \phi_h) \\
		& = b_1(\bu, \phi_h) + b_2(I^h_p p, \phi_h) - a_3(I^h_{\psi} \psi, \phi_h)
		= -b_2(e_p^I, \phi_h) + a_3(e_{\psi}^I, \phi_h).
		\end{align*}
		Differentiating the above equation with respect to time and taking $\phi_h = -e_{\psi}^A$, we can assert that
		\begin{align}
		- b_1(\partial_t e_{\bu}^A, e_{\psi}^A) - b_2(\partial_t e_p^A, e_{\psi}^A) + a_3(\partial_t e_{\psi}^A, e_{\psi}^A) = b_2(\partial_t e_p^I, e_{\psi}^A) -a_3(\partial_t e_{\psi}^I, e_{\psi}^A). \label{est-eq3}
		\end{align}
		Then we simply add \eqref{est-eq1}, \eqref{est-eq2} and \eqref{est-eq3}, to obtain
		\begin{align} \label{est-eq4} \begin{split}
		&a_1^h(e_{\bu}^A, \partial_t e_{\bu}^A)  + \tilde{a}_2^h(\partial_t e_p^A, e_p^A) +a_2^h(e_p^A, e_p^A) \\
		& + a_3(\partial_t e_{\psi}^A, e_{\psi}^A) - b_2(e_p^A, \partial_t e_{\psi}^A) - b_2(\partial_t e_p^A, e_{\psi}^A)\\
		& = (F- F^h)(\partial_{t} e_{\bu}^A ) + (\tilde{a}_2^h(\partial_t I^h_p p, e_p^A) - \tilde{a}_2(\partial_{t} p, e_p^A))
		\\& \quad + b_2(e_p^A, \partial_{t} e_{\psi}^I) + (G- G^h) (e_p^A)  + b_2(\partial_{t} e_p^I, e_{\psi}^A) - a_3(\partial_{t} e_{\psi}^I, e_{\psi}^A). \end{split}
		\end{align}
		Regarding the left-hand side of \eqref{est-eq4}, repeating arguments to obtain alike to \eqref{semi:stab}. That is,
		\begin{align*}
		& a_1^h(e_{\bu}^A, \partial_t e_{\bu}^A)  + \tilde{a}_2^h(\partial_t e_p^A, e_p^A)  \\
		& +a_2^h(e_p^A, e_p^A) + a_3(\partial_t e_{\psi}^A, e_{\psi}^A) - b_2(e_p^A, \partial_t e_{\psi}^A) - b_2(\partial_t e_p^A, e_{\psi}^A) \\
		&  \ge \frac{1}{2}\frac{\mathrm{d}}{\mathrm{d}t} a_1^h(e_{\bu}^A, e_{\bu}^A) + \frac{c_0}{2}\frac{\mathrm{d}}{\mathrm{d}t} \| e_p^A\|_0^2 + a_2^h(e_p^A, e_p^A) \\
		& \qquad \qquad + \frac{1}{\lambda} \sum_{K} \Bigl( \alpha^2 \bigl(\partial_{t} (\Pi^0_K e_p^A), \Pi^0_K e_p^A\bigr)_{0,K}  + \alpha^2 S^K_0  \bigl((I-\Pi^0_K) \partial_{t} e_p^A, (I-\Pi^0_K) e_p^A \bigr) \\
		& \qquad \qquad
		+ (\partial_{t} e_{\psi}^A, e_{\psi}^A)_{0,K} - \alpha (\Pi^0_K e_p^A, \partial_{t} e_{\psi}^A)_{0,K}
		- \alpha (\Pi^0_K \partial_{t} e_p^A, e_{\psi}^A)_{0,K} \Bigr)  \\
		&  \ge C \bigg( \mu\frac{\mathrm{d}}{\mathrm{d}t} \| \beps(e_{\bu}^A)\|_0^2 + c_0\frac{\mathrm{d}}{\mathrm{d}t} \| e_p^A\|_0^2 + \frac{2 \kappa_{\min}}{\eta} \|\nabla e_p^A\|_0^2  \\
		& \qquad \qquad
		+ \frac{1}{\lambda} \sum_{K} \biggl( \alpha^2\frac{\mathrm{d}}{\mathrm{d}t}\| (I-\Pi^0_K)  e_p^A\|_{0,K}^2
		+\frac{\mathrm{d}}{\mathrm{d}t} \|\alpha \Pi^0_K e_p^A -e_{\psi}^A\|_{0,K}^2  \biggr) \bigg).
		\end{align*}
		Then integrating equation \eqref{est-eq4} in time implies the bound
		\begin{align*}
		& \mu \| \beps(e_{\bu}^A(t))\|_0^2 + c_0  \| e_p^A (t)\|_0^2 + \frac{\kappa_{\min}}{\eta}\int_0^t \|\nabla e_p^A(s)\|_0^2 \, \mathrm{d}s \\
		& + \frac{1}{\lambda} \sum_{K} \Bigl(\alpha^2  \| (I-\Pi^0_K)  e_p^A(t)\|_{0,K}^2
		+  \|(\alpha \Pi^0_K e_p^A -e_{\psi}^A)(t)\|_{0,K}^2  \Bigr) \\
		& \lesssim \mu \| \beps(e_{\bu}^A(0))\|_0^2 + c_0  \| e_p^A (0)\|_0^2 \\
		& \quad + \frac{1}{\lambda} \sum_{K} \Bigl(\alpha^2  \| (I-\Pi^0_K)  e_p^A(0)\|_{0,K}^2 +  \|(\alpha \Pi^0_K e_p^A -e_{\psi}^A)(0)\|_{0,K}^2  \Bigr) \\
		& \quad + \underbrace{ \rho \int_0^t  \bigl((\bb- \bb^h)(s),\partial_{t} e_{\bu}^A(s) \bigr)_{0, \Omega} \, \mathrm{d}s}_{=:D_1} + \underbrace{\int_0^t \bigl((\ell - \ell^h)(s), e_p^A(s) \bigr)_{0, \Omega}\,  \mathrm{d}s}_{=:D_2} \\
		& \quad + \underbrace{\int_0^t \sum_K \Bigl(\tilde{a}_2^{h,K} \bigl(\partial_t(I^h_p p - p_{\pi})(s), e_p^A(s) \bigr) - \tilde{a}^{K}_2 \bigl(\partial_t(p - p_{\pi})(s), e_p^A(s) \bigr) \Bigr) \, \mathrm{d}s}_{=:D_3} \\
		& \quad + \underbrace{\int_0^t \Bigl(b_2 \bigl(e_p^A(s), \partial_{t} e_{\psi}^I(s) \bigr) + b_2 \bigl(\partial_{t} e_p^I (s), e_{\psi}^A (s)\bigr) - a_3 \bigl(\partial_{t} e_{\psi}^I(s), e_{\psi}^A(s) \bigr) \Bigr) \, \mathrm{d}s}_{=:D_4}.
		\end{align*}
		Then we can integrate by parts (also in time) and
		use  Cauchy-Schwarz inequality to arrive at
		\begin{align*}
		D_1 & = \rho \biggl(  \bigl((\bb-\bb^h)(t), e_{\bu}^A(t) \bigr)_{0, \Omega} - \bigl((\bb-\bb^h)(0),e_{\bu}^A(0) \bigr)_{0, \Omega} \\ & \qquad \qquad  + \int_0^t \bigl(\partial_{t} (\bb- \bb^h)(s),  e_{\bu}^A(s) \bigr)_{0, \Omega} \, \mathrm{d}s \biggr)\\
		& \le C_1(\rho)  h \biggl(|\bb(t)|_1 \| e_{\bu}^A(t)\|_0 + |\bb(0)|_1 \| e_{\bu}^A(0)\|_0 + \int_0^t | \partial_{t}\bb (s)|_{1} \| e_{\bu}^A (s)\|_{0}\, \mathrm{d}s \biggr),
		\end{align*}
		where we have used standard error estimate for the
		$L^2$-projection $\bPi_K^{0,0}$ onto piecewise constant functions.
		Using also Cauchy-Schwarz inequality and standard  error
		estimates for $\Pi_K^{0}$ on the term $D_2$ readily gives
		\begin{align*}
		D_2 \le  C_2 h \int_0^t |\ell (s)|_1 \|e_p^A(s)\|_0\, \mathrm{d}s.
		\end{align*}
		On the other hand, considering the polynomial approximation
		$p_{\pi}$ (cf. \eqref{poly_est_u}) of $p$ and utilising the triangle inequality yield
		\begin{align*}
		D_3 & \le C_3\bigg(c_0 + \frac{\alpha^2}{\lambda}\bigg) \\
		& \quad \times  \int_0^t \sum_K \Bigl(\| \partial_{t} (I^h_p p - p_{\pi})(s)\|_{0,K} + \| \partial_{t} (p - p_{\pi})(s)\|_{0,K} \Bigr)\|e_p^A(s)\|_{0,K} \, \mathrm{d}s \\
		& \le C_3 h^2 \bigg(c_0 + \frac{\alpha^2}{\lambda}\bigg) \int_0^t |\partial_{t} p(s)|_2 \|e_p^A(s)\|_0\, \mathrm{d}s.
		\end{align*}
		Also,
		\begin{align*}
		D_4 & = \int_0^t \Bigl(b_2 \bigl(e_p^A(s), \partial_{t} e_{\psi}^I(s)\bigr) + b_2 \bigl(\partial_{t} e_p^I (s), e_{\psi}^A (s)\bigr) - a_3 \bigl(\partial_{t} e_{\psi}^I(s), e_{\psi}^A(s)\bigr) \Bigr) \, \mathrm{d}s \\
		& \le \frac{1}{\lambda} \int_0^t  \Bigl(\alpha \| e_p^A (s)\|_0 \| \partial_{t} e_{\psi}^I(s) \|_0+  \bigl(\alpha \|\partial_{t} e_p^I(s) \|_0 + \| \partial_{t} e_{\psi}^I (s)\|_0 \bigr) \|e_{\psi}^A (s)\|_0 \Bigr)\, \mathrm{d}s \\
		& \le  \frac{C_4}{\lambda}   h \int_0^t  \big( \alpha \| e_p^A (s)\|_0 (| \partial_{t} \psi(s) |_1 +
		|\partial_t \bu(s)|_2 ) + (\alpha h |\partial_{t} p(s) |_2 + | \partial_{t} \psi (s)|_1 \\
		& \qquad \qquad \qquad \qquad  + |\partial_t \bu(s)|_2) \|e_{\psi}^A (s)\|_0 \bigr) \, \mathrm{d}s.
		\end{align*}
		Using \eqref{discr-infsup} and a combination of equations \eqref{weak-Iuh},
		\eqref{weak-uh} and \eqref{weak-u}, we get
		\begin{align} \label{bound:semi_inf-sup}
		\begin{split}
		\| e_{\psi}^A(t) \|_0 & \le \sup_{\bv_h \in \bV_h} \frac{b_1(\bv_h,  e_{\psi}^A(t))}{\| \bv_h \|_1} \le  C_5 \left( \rho \sum_K \| (\bb - \bb^h)(t)\|_{0,K} + \mu \| \beps(e_{\bu}^A(t))\|_0 \right) \\
		& \le C_5 \bigl( \rho \; h | \bb(t) |_1 + \mu \| \beps(e_{\bu}^A(t))\|_0 \bigr).
		\end{split}
		\end{align}
		Then the bound of~$D_4$ becomes
		\begin{align*}
		D_4 \le  \frac{C_6}{\lambda}   h \int_0^t  \Bigl( &(\alpha h |\partial_{t} p(s) |_2
		+ | \partial_{t} \psi (s)|_1+|\partial_t \bu(s)|_2) ( \rho h |\bb(s)|_1 + \mu \| \beps(e_{\bu}^A(t))\|_0)  \\
		& + \alpha \| e_p^A (s)\|_0 (| \partial_{t} \psi(s) |_1
		+|\partial_t \bu(s)|_2) \Bigr) \, \mathrm{d}s.
		\end{align*}
		Combining the bounds of all $D_i, i=1,2,3,4$ implies that
		\begin{align*}
		&  \mu \| \beps(e_{\bu}^A(t))\|_0^2 + c_0 \| e_p^A (t)\|_0^2 + \frac{\kappa_{\min}}{\eta} \int_0^t \| \nabla e_p^A (s)\|_0^2  \, \mathrm{d} s
		\\ & + \frac{1}{\lambda} \sum_{K} \Bigl( \alpha^2 \| (I - \Pi^0_K) e_p^A (t)\|_{0,K}^2 +  \| (\alpha \Pi^0_K e_p^A - e_{\psi}^A)(t) \|_{0,K}^2\Bigr)  \\
		& \le \mu \| \beps(e_{\bu}^A(0))\|_0^2 + \Big(c_0 + \frac{\alpha^2}{\lambda} \Big)  \| e_p^A (0)\|_0^2 +  \frac{1}{\lambda} \|e_{\psi}^A(0)\|_0^2  \\
		& \quad + \frac{\mu}{2} \| \beps(e_{\bu}^A(t))\|_0^2 + C \,h \bigg(h |\bb(t)|_1^2  + |\bb(0)|_1 \| e_{\bu}^A(0)\|_0 \\
		& \qquad \qquad + \int_0^t \biggl( | \partial_{t}\bb (s)|_{1} + h |\partial_{t} p(s) |_2 + | \partial_{t} \psi (s)|_1 +|\partial_t \bu(s)|_2\biggr) \mu \| \beps(e_{\bu}^A (s))\|_{0}\, \mathrm{d}s  \\
		& \qquad \qquad  + \int_0^t \biggl( |\ell (s)|_1 + h |\partial_{t} p(s) |_2 + | \partial_{t} \psi(s) |_1 +|\partial_t \bu(s)|_2 \biggr) \|e_p^A(s)\|_0\, \mathrm{d}s \\ & \qquad \qquad + h \int_0^t
		\bigl( h |\partial_{t} p(s) |_2 + | \partial_{t} \psi (s)|_1 +|\partial_t \bu(s)|_2\bigr) |\bb(s)|_1 \, \mathrm{d}s  \biggr).
		\end{align*}
		The Poincar\'e, Young's inequalities and Gronwall lemma now allows us to conclude that
		\begin{align*}
		&  \mu \| \beps(e_{\bu}^A(t))\|_0^2 + c_0 \| e_p^A (t)\|_0^2 + \frac{\kappa_{\min}}{\eta} \int_0^t \| \nabla e_p^A (s)\|_0^2  \, \mathrm{d} s
		\\
		& \le \mu \| \beps(e_{\bu}^A(0))\|_0^2 + \Big(c_0 + \frac{\alpha^2}{\lambda} \Big)  \| e_p^A (0)\|_0^2 +  \frac{1}{\lambda} \|e_{\psi}^A(0)\|_0^2  \\
		& \quad + C \,h^2 \bigg( \sup_{t \in [0,t_{\text{final}}] } |\bb(t)|_1^2 + \int_0^t \biggl( |\bb(s)|_1^2 + | \partial_{t}\bb (s)|_{1}^2 + |\ell (s)|_1^2 \\
		& \qquad \qquad +  | \partial_{t} \psi (s)|_1^2 +|\partial_t \bu(s)|_2^2  + h^2 |\partial_{t} p(s) |_2^2 \biggr) \, \mathrm{d}s \biggr).
		\end{align*}
	    Then choosing $\bu_h(0): =\bu_I(0)$,
		$\psi_h(0): = \Pi^{0,0}\psi(0)$, $p_h(0): = p_I(0)$ and applying
		the triangle inequality together with \eqref{bound:semi_inf-sup} completes the rest of the proof.
		\end{proof}
	
	\begin{theorem}[Fully-discrete error estimates]
		Let ($\bu(t),p(t),\psi(t)$) $\in \bV \times Q \times Z$
		and ($\bu_h^n,p_h^n,\psi_h^n$) $\in \bV_h \times Q_h \times Z_h$ be the unique
		solutions to problems \eqref{weak-u}-\eqref{weak-psi} and \eqref{weak-uhn}-\eqref{weak-psihn},
		respectively. Then the following estimates hold for any $n=1,\ldots,N$, with constants~$C$ independent of~$h,\, \Delta t ,\, \lambda$:
			\begin{align} \label{th4.3-est}
			\begin{split}
			& \mu \| \beps(\bu(t_n) - \bu_h^n) \|_0^2 \, + \, \| \psi(t_n) - \psi_h^n \|_0^2  \\
			& \qquad \quad + (\Delta t) \frac{\kappa_{\min}}{\eta}\| \nabla (p(t_n) - p_h^n)\|_0^2   \, \le C\, (h^2 + \Delta t^2).
			\end{split}
			\end{align}
	\end{theorem}
	\begin{proof}
		As done for the semidiscrete case, we split the individual errors as
		\begin{align*}
		\bu(t_n) - \bu_h^n & = (\bu(t_n) - I^h_{\bu} \bu(t_n)) + (I^h_{\bu} \bu(t_n)- \bu_h^n):= E_{\bu}^{I,n} + E_{\bu}^{A,n}, \\
		\psi(t_n) - \psi_h^n &= (\psi(t_n) - I^h_{\psi} \psi(t_n)) + (I^h_{\psi} \psi(t_n)- \psi_h^n):= E_{\psi}^{I,n} + E_{\psi}^{A,n}, \\
		p(t_n) - p_h^n &= (p(t_n) - I^h_{p} p(t_n)) + (I^h_{p} p(t_n)- p_h^n):= E_{p}^{I,n} + E_{p}^{A,n}.
		\end{align*}
		Then, from estimate \eqref{estimate-Ihu} we have
		\begin{align} \label{estimate-E_u^I}
		\|E_{\bu}^{I,n} \|_1 &\le C h (|\bu(t_n)|_2 + |\psi(t_n)|_1) \nonumber \\
		& \le C h ( | \bu(0) |_2 + | \psi(0) |_1 + \|\partial_t\bu\|_{\bL^1(0,t_n; \bH^2(\Omega))} + \|\partial_t \psi\|_{L^1(0,t_n; H^1(\Omega))} ).
		\end{align}
		Following the same steps as before, we get
		\begin{align} \label{estimate-E_{psi}^I}
		\|E_{\psi}^{I,n} \|_0 & \le C h ( | \bu(0) |_2 + | \psi(0) |_1 + \|\partial_t\bu\|_{\bL^1(0,t_n; \bH^2(\Omega))} + \|\partial_t\psi\|_{L^1(0,t_n; H^1(\Omega))} ),\\
		\label{estimate-E_{p}^I}
		\|E_{p}^{I,n} \|_1 &\le C h ( | p(0) |_2 + \|\partial_t p\|_{L^1(0,t_n; H^2(\Omega))}).
		\end{align}
		From equations  \eqref{weak-Iuh}, \eqref{weak-uhn} and \eqref{weak-u}, we readily get
		\begin{align} \label{estimate-eq1}
		a_1^h(E_{\bu}^{A,n},\bv_h) + b_1(\bv_h, E_{\psi}^{A,n}) = F^n(\bv_h) - F^{h,n}(\bv_h).
		\end{align}
		Now, use of \eqref{weak-Ipsih}, \eqref{eq:psihn}
		and differentiating \eqref{weak-psi} with respect to time implies
		\begin{align} \label{estimate-eq2}
		&b_1(E_{\bu}^{A,n} - E_{\bu}^{A,n-1}, \phi_h) + b_2(E_p^{A,n}-E_p^{A,n-1}, \phi_h) - a_3(E_{\psi}^{A,n} - E_{\psi}^{A,n-1}, \phi_h) \nonumber \\
		& \quad = b_1((\bu(t_n) - \bu(t_{n-1})) - (\Delta t) \partial_{t} \bu(t_n), \phi_h) \nonumber\\
		& \qquad + b_2((I_p^h p(t_n) - I_p^h p(t_{n-1})) - (\Delta t) \partial_{t} p(t_n), \phi_h) \nonumber \\
		& \qquad - a_3((I_{\psi}^h \psi(t_n) - I_{\psi}^h \psi(t_{n-1})) - (\Delta t) \partial_{t} \psi (t_n), \phi_h).
		\end{align}
		Choosing $\bv_h = E_{\bu}^{A,n} - E_{\bu}^{A,n-1 } $ in \eqref{estimate-eq1} and $\phi_h = - E_{\psi}^{A,n}$ in \eqref{estimate-eq2} then adding the outcomes, we get
		\begin{align} \label{estimate-eq3}
		& a_1^h(E_{\bu}^{A,n}, E_{\bu}^{A,n}- E_{\bu}^{A,n-1}) + a_3(E_{\psi}^{A,n} - E_{\psi}^{A,n-1}, E_{\psi}^{A,n}) - b_2(E_{p}^{A,n} - E_{p}^{A,n-1}, E_{\psi}^{A,n}) \nonumber \\
		& = \rho ( \bb(t_n) - \bb^n_h, E_{\bu}^{A,n}- E_{\bu}^{A,n-1 } )_{0, \Omega} \nonumber\\
		& \qquad  - b_1((\bu(t_n) - \bu(t_{n-1})) - (\Delta t) \partial_{t} \bu(t_n), E_{\psi}^{A,n}) \nonumber\\
		& \qquad - b_2((I_p^h p(t_n) - I_p^h p(t_{n-1})) - (\Delta t) \partial_{t} p(t_n), E_{\psi}^{A,n})\nonumber \\
		& \qquad + a_3((I_{\psi}^h \psi(t_n) - I_{\psi}^h \psi(t_{n-1})) - (\Delta t) \partial_{t} \psi (t_n), E_{\psi}^{A,n}).
		\end{align}
		Next, the use of \eqref{weak-Iph}, \eqref{weak-ph} and \eqref{weak-p} with $q_h = E_p^{A,n}$, readily gives
		\begin{align} \label{estimate-eq4}
		& \tilde{a}_2^h(E_{p}^{A,n} - E_{p}^{A,n-1}, E_{p}^{A,n}) +  \Delta t  a_2^h(E_{p}^{A,n}, E_{p}^{A,n}) - b_2(E_{p}^{A,n}, E_{\psi}^{A,n} - E_{\psi}^{A,n-1}) \nonumber
		\\ & = \Delta t ( \ell(t_n)- \ell^n_h, E_{p}^{A,n} )_{0, \Omega} + \tilde{a}_2^h( I^h_p p(t_n) - I^h_p p(t_{n-1}), E_{p}^{A,n}) \\
		& \qquad -  \tilde{a}_2((\Delta t) \partial_{t} p(t_n), E_{p}^{A,n}) + b_2(E_{p}^{A,n}, (\Delta t) \partial_{t} \psi - (I^h_{\psi} \psi(t_n)) - I^h_{\psi} \psi(t_{n-1})),
		\end{align}
		and adding the resulting equations \eqref{estimate-eq3}- \eqref{estimate-eq4} we can write
		\begin{align*}
		& a_1^h(E_{\bu}^{A,n}, E_{\bu}^{A,n}- E_{\bu}^{A,n-1 }) + a_3(E_{\psi}^{A,n} - E_{\psi}^{A,n-1}, E_{\psi}^{A,n}) -  b_2(E_{p}^{A,n} - E_{p}^{A,n-1}, E_{\psi}^{A,n})\\
		& -  b_2(E_{p}^{A,n}, E_{\psi}^{A,n} - E_{\psi}^{A,n-1}) + \tilde{a}_2^h(E_{p}^{A,n} - E_{p}^{A,n-1}, E_{p}^{A,n}) +  \Delta t  a_2^h(E_{p}^{A,n}, E_{p}^{A,n})  \\
		& = \rho ( \bb(t_n) - \bb^n_h, E_{\bu}^{A,n}- E_{\bu}^{A,n-1} )_{0, \Omega} + \Delta t ( \ell(t_n)- \ell^n_h, E_{p}^{A,n} )_{0, \Omega} \\
		& \qquad - b_1((\bu(t_n) - \bu(t_{n-1})) - (\Delta t) \partial_{t} \bu(t_n), E_{\psi}^{A,n}) \nonumber\\
		& \qquad - b_2((I_p^h p(t_n) - I_p^h p(t_{n-1})) - (\Delta t) \partial_{t} p(t_n), E_{\psi}^{A,n}) \\
		& \qquad +  a_3((I_{\psi}^h \psi(t_n) - I_{\psi}^h \psi(t_{n-1})) - (\Delta t) \partial_{t} \psi (t_n), E_{\psi}^{A,n})  \\
		& \qquad + \tilde{a}_2^h( I^h_p p(t_n) - I^h_p p(t_{n-1}), E_{p}^{A,n}) - \tilde{a}_2((\Delta t) \partial_{t} p(t_n), E_{p}^{A,n}) \\
		& \qquad + b_2(E_{p}^{A,n}, (\Delta t) \partial_{t} \psi - (I^h_{\psi} \psi(t_n)) - I^h_{\psi} \psi(t_{n-1})),
		\end{align*}
		and we will repeat the arguments identical to \eqref{semi:stab} to get		
		\begin{align*}
		& a_3(E_{\psi}^{A,n} - E_{\psi}^{A,n-1}, E_{\psi}^{A,n}) -  b_2(E_{p}^{A,n} - E_{p}^{A,n-1}, E_{\psi}^{A,n})\\
		& \qquad -  b_2(E_{p}^{A,n}, E_{\psi}^{A,n} - E_{\psi}^{A,n-1}) + \tilde{a}_2^h(E_{p}^{A,n} - E_{p}^{A,n-1}, E_{p}^{A,n}) \\
		& = (\Delta t) \bigg( c_0(\delta_t E_{p}^{A,n} , E_{p}^{A,n})_{0,\Omega} + \frac{1}{\lambda} \sum_{K} \big(\alpha^2 (\delta_t (I- \Pi^0_K) E_p^{A,n}, (I- \Pi^0_K) E_p^{A,n})_{0,K} \\
		& \qquad -(\delta_t (\alpha \Pi^0_K E_p^{A,n} - E_{\psi}^{A,n}), \alpha \Pi^0_K E_p^{A,n} - E_{\psi}^{A,n})_{0,K} \big) \bigg),
		\end{align*}
		The left-hand side can be bounded by using the inequality \eqref{bound:discrete_ineq} and then summing over $n$ we get
		\begin{align} \label{estimate-eq5}
		& \mu (\| \beps(E_{\bu}^{A,n})\|_0^2 - \| \beps(E_{\bu}^{A,0})\|_0^2) + c_0 (\| E_{p}^{A,n}\|_0^2 - \| E_{p}^{A,0}\|_0^2) + (\Delta t) \frac{\kappa_{\min}}{\eta} \sum_{j=1}^n \| \nabla E_{p}^{A,j} \|_0^2 \nonumber \\
		& \qquad + (1/\lambda) \bigg( \alpha^2 (\| (I- \Pi^0_K) E_p^{A,n} \|_0^2 - \| (I- \Pi^0_K) E_p^{A,0} \|_0^2) \nonumber \\
		& \qquad + \sum_{K} (\| \alpha \Pi^0_K E_p^{A,n} - E_{\psi}^{A,n}\|_0^2 - \| \alpha \Pi^0_K E_p^{A,0} - E_{\psi}^{A,0} \|_0^2 ) \bigg) \\
		& \le \underbrace{ \sum_{j=1}^n \rho ( \bb(t_j) - \bb^j_h, E_{\bu}^{A,j}- E_{\bu}^{A,j-1} )_{0,\Omega}}_{:=L_1}
		+ \underbrace{ \sum_{j=1}^n \Delta t ( \ell(t_j)- \ell^j_h, E_{p}^{A,j} )_{0,\Omega}}_{:=L_2} \nonumber \\
		& \qquad - \underbrace{ \sum_{j=1}^n b_1((\bu(t_n) - \bu(t_{n-1})) - (\Delta t) \partial_{t} \bu(t_n), E_{\psi}^{A,n})}_{:=L_3} \nonumber\\
		& \qquad - \underbrace{ \sum_{j=1}^n b_2((I_p^h p(t_j) - I_p^h p(t_{j-1})) - (\Delta t) \partial_{t} p(t_j), E_{\psi}^{A,j})}_{:=L_4} \nonumber\\
		& \qquad + \underbrace{ \sum_{j=1}^n a_3((I_{\psi}^h \psi(t_j) - I_{\psi}^h \psi(t_{j-1})) - (\Delta t) \partial_{t} \psi (t_j), E_{\psi}^{A,j}) }_{:=L_5}  \nonumber\\
		& \qquad + \underbrace{ \sum_{j=1}^n (\tilde{a}_2^h( I^h_p p(t_j) - I^h_p p(t_{j-1}), E_{p}^{A,j}) -  \tilde{a}_2((\Delta t) \partial_{t} p(t_j), E_{p}^{A,j}) )}_{:=L_6} \nonumber \\
		& \qquad + \underbrace{ \sum_{j=1}^n b_2(E_{p}^{A,j}, (\Delta t) \partial_{t} \psi - (I^h_{\psi} \psi(t_j) - I^h_{\psi} \psi(t_{j-1}))}_{:=L_7}. \nonumber
		\end{align}
		We bound the term $L_1$ with the formula \eqref{eq:discrete_equ}, the estimates of projection $\bPi^{0,0}_K$ the Taylor expansion and generalised Young's inequality,
		\begin{align*}
		L_1 &= \rho \big(((\bb - \bb_h)(t_n), E_{\bu}^{A,n})_{0,\Omega} - ((\bb - \bb_h)(0), E_{\bu}^{A,0})_{0,\Omega} \\
		& \qquad - \sum_{j=1}^{n} (\Delta t)( \delta_t(\bb - \bb_h)(t_j), E_{\bu}^{A,j-1})_{0,\Omega}) \big) \\
		& = \rho \bigg( ((\bb - \bb_h)(t_n), E_{\bu}^{A,n})_{0,\Omega} - ((\bb - \bb_h)(0), E_{\bu}^{A,0})_{0,\Omega} \\
		& \qquad \quad - \sum_{j=1}^{n} (\Delta t)( \partial_t(\bb - \bb_h)(t_j), E_{\bu}^{A,j-1})_{0,\Omega}  \\
		& \qquad \quad + \sum_{j=1}^{n} \Big( \int_{t_{j-1}}^{t_j} (s-t_{j-1})  \partial_{tt} (\bb -\bb_h)(s)\, \mathrm{d}s, E_{\bu}^{A,j-1}  \Big)_{0,\Omega}  \bigg)\\
		& \le \frac{\mu}{2} \| \beps (E_{\bu}^{A,n})\|_0^2 + \mu \| \beps (E_{\bu}^{A,0})\|_0^2 + C_1 \Big( \frac{\rho^2}{\mu}h^2  \big( \max_{0 \le j \le n} |\bb(t_j) |_1^2 + \Delta t \sum_{j=1}^n | \partial_{t} \bb |_1^2 \big) \\
		& \qquad \qquad \qquad \quad + (\Delta t) \sum_{j=0}^{n-1} \mu \| \beps(E_{\bu}^{A,j}) \|_0^2 + \frac{\rho^2}{\mu} (\Delta t)^2 h^2 \int_0^T | \partial_{tt} \bb (s)|_1^2 \, \mathrm{d}s \Big).
		\end{align*}
		Then the estimate of projection $\Pi^0_K$, Poincar\'e and Young's inequalities gives
		\begin{align*}
		L_2 & \le C_2 \sum_{j=1}^n (\Delta t) h^2 | \ell(t_j)|_2 \|\nabla E_{p}^{A,j}\|_0 \\
		& \le C_2 \sum_{j=1}^n (\Delta t) \frac{\eta }{\kappa_{\min}} h^4 | \ell(t_j)|_2^2 +  (\Delta t) \frac{\kappa_{\min}}{6 \eta} \sum_{j=1}^n \| \nabla E_{p}^{A,j}\|_0^2.
		\end{align*}
		The  discrete inf-sup condition \eqref{discr-infsup} yields
			\begin{align} \label{bound:discrete-inf-sup}
			\|E_{\psi}^{A,j} \|_0 \le C ( h |\bb(t_j)|_1 + \|\beps(E_{\bu}^{A,j}) \|_0 ).
			\end{align}
			Applying Taylor series expansion together with \eqref{bound:discrete-inf-sup},  the Cauchy Schwarz and Young's inequalities enable us
			\begin{align*}
			L_3 & \le C \sum_{j=1}^n \| ((\bu(t_j) - \bu(t_{j-1})) - (\Delta t) \partial_{t} \bu (t_j))) \|_0 (h |\bb(t_j)|_1 +  \|\beps(E_{\bu}^{A,j}) \|_0 ) \\
			& \le C \Big( (\Delta t )^2 \| \partial_{tt} \bu \|_{\bL^2(0,t_n;\bL^2(\Omega))}^2 + (\Delta t) \sum_{j=1}^n \Big( \rho^2 h^2 |\bb(t_j)|_1^2 + \mu \|\beps(E_{\bu}^{A,j}) \|_0^2 \Big).
			\end{align*}
		
		By use of estimates of the projection $I_p^h$, \eqref{bound:discrete-inf-sup}, the Cauchy Schwarz and Young's inequalities we get
		\begin{align*}
		L_4 & \le C_3 \frac{\alpha}{\lambda} \sum_{j=1}^n \Big( \| I_p^h (p(t_j) - p(t_{j-1})) - (p(t_j) - p(t_{j-1})) \|_0 \\
		& \qquad \qquad \qquad+ \| (p(t_j) - p(t_{j-1})) - (\Delta t) \partial_{t} p(t_j) \|_0 \Big)   \|E_{\psi}^{A,j} \|_0  \\
		& \le C_3 \frac{\alpha}{\lambda} \sum_{j=1}^n \Big( h^2 | p(t_j) - p(t_{j-1}) |_2 + \lVert \int_{t_{j-1}}^{t_j}(s-t_{j-1}) \partial_{tt} p(s)\, \mathrm{d}s \rVert_0 \Big) \|E_{\psi}^{A,j} \|_0
		\\
		& \le C_3 \frac{\alpha}{\lambda} \sum_{j=1}^n \bigg( h^2  \Big( (\Delta t) \int_{t_{j-1}}^{t_j}| \partial_{t}p(s)|_2^2\, \mathrm{d}s \Big)^{1/2} \\
		& \qquad \qquad \qquad+ \Big( (\Delta t )^{3} \int_{t_{j-1}}^{t_j} \| \partial_{tt} p(s ) \|_0^2\, \mathrm{d}s\Big)^{1/2} \bigg) \|E_{\psi}^{A,j} \|_0\\
		&  \le C_3 \Big(\frac{\alpha}{\lambda} \Big)^2  (1+ \mu)  \sum_{j=1}^n \bigg(  h^4  \Big(\int_{t_{j-1}}^{t_j}| \partial_{t}p(s) |_2^2\, \mathrm{d}s \Big)^2
		+ (\Delta t )^3 \| \partial_{tt} p \|_{L^2(0,t_n;L^2(\Omega))}^2 \bigg) \\
		& \qquad +  \rho^2 h^2(\Delta t ) \sum_{j=1}^n |\bb(t_j)|_1^2 + \mu (\Delta t ) \sum_{j=1}^n \|\beps(E_{\bu}^{A,j}) \|_0^2 \\
		& \le C \bigg(  h^4 \| \partial_{t}p \|_{L^2(0,t_n;H^2(\Omega))}^2
		+ (\Delta t )^2 \| \partial_{tt} p \|_{L^2(0,t_n;L^2(\Omega))}^2 \bigg) \\
		& \quad +  (\Delta t) \sum_{j=1}^n \Big( \rho^2 h^2 |\bb(t_j)|_1^2 + \mu  \|\beps(E_{\bu}^{A,j}) \|_0^2 \Big).
		\end{align*}
		The stability of $a_3(\cdot, \cdot)$ and the proof for the bound of $L_4$ gives
		\begin{align*}
		L_5 &\le  (1/ \lambda ) \sum_{j=1}^n \| (I_{\psi}^h \psi(t_j) - I_{\psi}^h \psi(t_{j-1})) - (\Delta t) \partial_{t} \psi (t_j) \|_0 ( \rho h |\bb(t_j)|_1 +  \|\beps(E_{\bu}^{A,j}) \|_0 ) \\
		& \le C \bigg(  h^2 (\| \partial_{t} \psi \|_{L^2(0,t_n;H^1(\Omega))}^2 + \| \partial_{t} \bu\|_{\bL^2(0,t_n;\bH^2(\Omega))}^2)
		+ (\Delta t )^2 \| \partial_{tt} \psi \|_{L^2(0,t_n;L^2(\Omega))}^2 \bigg) \\
		& \quad +  (\Delta t) \sum_{j=1}^n \Big( \rho^2 h^2 |\bb(t_j)|_1^2 + \mu  \|\beps(E_{\bu}^{A,j}) \|_0^2 \Big).
		\end{align*}
		The polynomial approximation $p_\pi$ for fluid pressure, stability of the bilinear forms $\tilde{a}_2(\cdot, \cdot ), \tilde{a}_2^h(\cdot, \cdot )$, the Cauchy Schwarz, Poincar\'e and Young's inequalitites gives		
		\begin{align*}
		L_6 & = \sum_{j=1}^n \Big(\tilde{a}_2^h( (I^h_p p(t_j) - I^h_p p(t_{j-1})) - (p_{\pi}(t_j) - p_{\pi}(t_{j-1})), E_{p}^{A,j})  \\
		&  \qquad \quad + \tilde{a}_2((p_{\pi}(t_j) - p_{\pi}(t_{j-1})) - (p(t_j) - p(t_{j-1})), E_{p}^{A,j})  \\
		& \qquad \quad +   \tilde{a}_2( (p(t_j) - p(t_{j-1})) - (\Delta t) \partial_{t} p(t_j), E_{p}^{A,j}) \Big) \\
		& \le C_5 \Big( c_0 +  \frac{\alpha^2}{\lambda} \Big) \sum_{j=1}^n \bigg( h^2 \Big( (\Delta t)  \int_{t_{j-1}}^{t_j} |\partial_t p(s) |_2^2 \, \mathrm{d}s \Big)^{1/2} \\
		& \qquad \qquad \qquad \qquad  \qquad \quad + \Big((\Delta t)^3\int_{t_{j-1}}^{t_j} \|  \partial_{tt} p(s) \|_0^2\, \mathrm{d}s \Big)^{1/2} \bigg) \| \nabla E_{p}^{A,j} \|_0 \\
		& \le C \sum_{j=1}^n \Big( h^4 \|\partial_t p \|_{L^2(0,t_n;H^2(\Omega))}^2 + (\Delta t)^2 \|  \partial_{tt} p \|_{L^2(0,t_n;L^2(\Omega))} \Big)  \\
		& \qquad + (\Delta t) \frac{\kappa_{\min}}{6 \eta} \sum_{j=1}^n \| \nabla E_{p}^{A,j} \|_0^2.
		\end{align*}
		The continuity of $b_2(\cdot, \cdot)$ and the bound of the $L_5$ gives 	
		\begin{align*}
		L_7 & \le  \Big(\frac{\alpha}{\lambda} \Big)  \sum_{j=1}^n \| (\Delta t) \partial_{t} \psi(t_j) - (I^h_{\psi} \psi(t_j) - I^h_{\psi} \psi(t_{j-1}))\|_0  \| E_{p}^{A,j}\|_0 \\
		& \le C \bigg(  h^2 (\| \partial_{t} \psi \|_{L^2(0,t_n;H^1(\Omega))}^2 + \| \partial_{t} \bu \|_{\bL^2(0,t_n;\bH^2(\Omega))}^2)+ (\Delta t )^2 \| \partial_{tt} \psi \|_{L^2(0,t_n;L^2(\Omega))}^2 \bigg)  \\
		& \quad
		+(\Delta t) \frac{\kappa_{\min}}{6 \eta} \sum_{j=1}^n \| \nabla E_p^{A,j} \|_0^2.
		\end{align*}
		The bounds of all $L_i$'s, $i=1, \dots, 7$ implies
		\begin{align*}
		&\mu \| \beps(E_{\bu}^{A,n})\|_0^2 + c_0 \| E_{p}^{A,n}\|_0^2 + (\Delta t) \frac{\kappa_{\min}}{\eta} \sum_{j=1}^n \| \nabla E_{p}^{A,j} \|_0^2 \\
		& \le \frac{\mu}{2} \| \beps (E_u^{A,n})\|_0^2 + (\Delta t) \frac{\kappa_{\min}}{2 \eta}\sum_{j=1}^n \| \nabla E_p^{A,j} \|_0^2 + C (\Delta t) \sum_{j=0}^{n} \mu \| \beps(E_u^{A,j}) \|_0^2 \\
		& \quad + C \Big( \| \beps(E_{\bu}^{A,0})\|_0^2 +  \| E_{p}^{A,0}\|_0^2 + \| E_{\psi}^{A,0}\|_0^2  \\
		& \qquad \qquad +  \Big(1 + \Delta t \Big) h^2   \max_{0 \le j \le n} |\bb(t_j) |_1^2 + h^2 \Delta t \sum_{j=1}^n (|\bb(t_j)|_1^2 + | \partial_{t} \bb |_1^2 )\\
		& \qquad  \qquad + h^2 (\Delta t)^2 \int_0^T | \partial_{tt} \bb (s)|_1^2 \, ds +  h^4 (\Delta t) \sum_{j=1}^n | \ell(t_j)|_2^2 \\
		& \qquad  \qquad
		+ (\Delta t )^2 \big(\| \partial_{tt} p \|_{L^2(0,t_n;L^2(\Omega))}^2  + \| \partial_{tt} \bu \|_{\bL^2(0,t_n;\bL^2(\Omega))}^2 + \| \partial_{tt} \psi \|_{L^2(0,t_n;L^2(\Omega))}^2 \big)\\
		& \qquad  \qquad +  h^2 \big(\| \partial_{t} \psi \|_{L^2(0,t_n;H^1(\Omega))}^2 + \| \partial_{t} \bu\|_{\bL^2(0,t_n;\bH^2(\Omega))}^2
		+ h^2 \| \partial_{t}p \|_{L^2(0,t_n;H^2(\Omega))}^2 \big) \bigg).
		\end{align*}
		The discrete Gronwall's inequality concludes that
		\begin{align*}
		&\mu \| \beps(E_{\bu}^{A,n})\|_0^2 + c_0 \| E_{p}^{A,n}\|_0^2 + (\Delta t) \frac{\kappa_{\min}}{\eta} \sum_{j=1}^n \| \nabla E_{p}^{A,j} \|_0^2 \\
		& \le C \bigg( \mu \| \beps(E_{\bu}^{A,0})\|_0^2 + (c_0 + \alpha^2/\lambda) \| E_{p}^{A,0}\|_0^2 + (1/\lambda) \| E_{\psi}^{A,0}\|_0^2  \\
		& \qquad  + \Big(1 + \Delta t \Big) h^2   \max_{0 \le j \le n} |\bb(t_j) |_1^2 + h^2 \Delta t \sum_{j=1}^n | \partial_{t} \bb |_1^2 \\
		& \qquad + (\Delta t)^2 h^2 \int_0^T | \partial_{tt} \bb (s)|_1^2 \, ds + h^2 (\Delta t) \sum_{j=1}^n |\bb(t_j)|_1^2 +  h^4 (\Delta t) \sum_{j=1}^n | \ell(t_j)|_2^2 \\
		&  \qquad + (\Delta t )^2 \big(\| \partial_{tt} \bu \|_{\bL^2(0,t_n;\bL^2(\Omega))}^2 + \| \partial_{tt} \psi \|_{L^2(0,t_n;L^2(\Omega))}^2 + \| \partial_{tt} p \|_{L^2(0,t_n;L^2(\Omega))}^2 \big)\\
		& \qquad +  h^2 \big( \| \partial_{t} \bu\|_{\bL^2(0,t_n;\bH^2(\Omega))}^2
		+ \| \partial_{t} \psi \|_{L^2(0,t_n;H^1(\Omega))}^2 + h^2 \| \partial_{t}p \|_{L^2(0,t_n;H^2(\Omega))}^2 \big) \bigg).
		\end{align*}
		Now the desire result \eqref{th4.3-est}  holds after  choosing $\bu_h^0: =\bu_I(0)$,
			$\psi_h^0: = \Pi^{0,0}\psi(0)$, $p_h^0: = p_I(0)$ and applying triangle's inequality together with \eqref{bound:discrete-inf-sup}.
	\end{proof}

	\section{Numerical results} \label{sec:results}
	In this section conduct numerical tests to  computationally  reconfirm  the convergence rates of the proposed virtual element scheme
	and present one test of applicative interest in poromechanics. All numerical results are produced by an
	in-house MATLAB code, using sparse factorisation as linear solver.
	
	\begin{figure}[t]
		\begin{center}
			\subfigure[]{\includegraphics[width=0.3\textwidth]{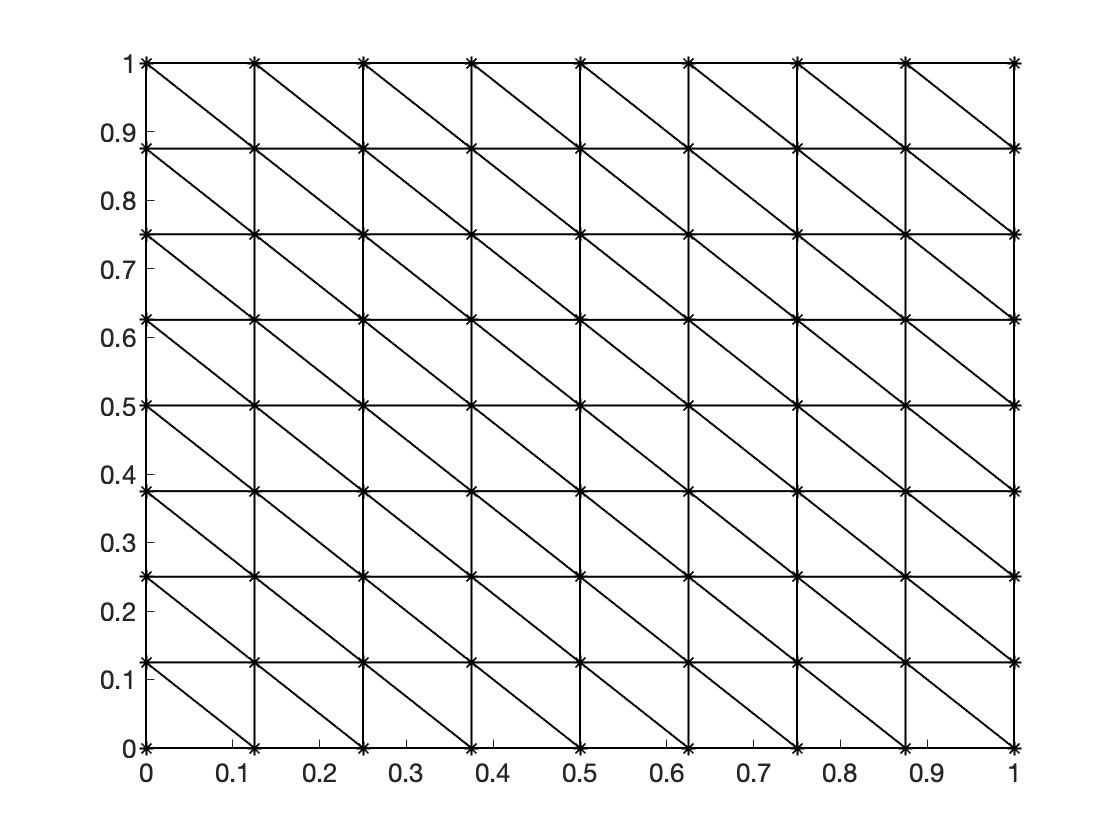}}
			\subfigure[]{\includegraphics[width=0.3\textwidth]{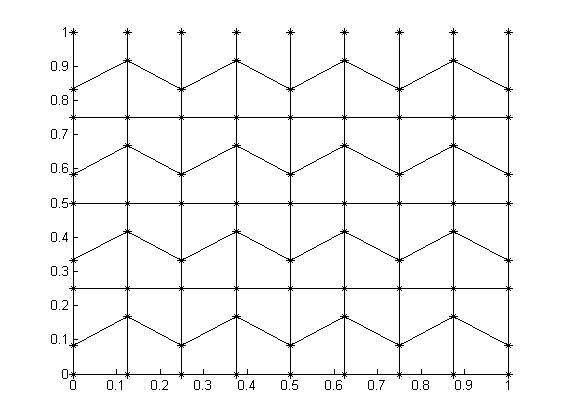}}
			\subfigure[]{\includegraphics[width=0.3\textwidth]{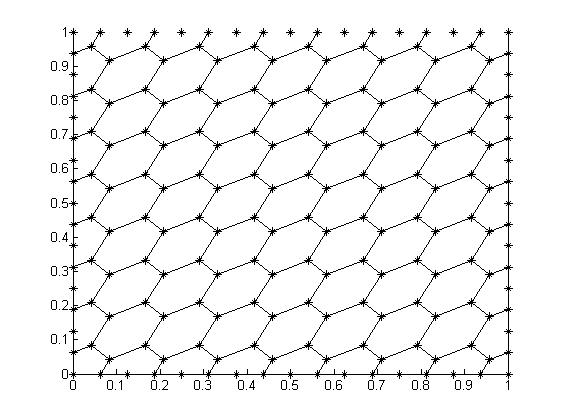}}
		\end{center}		
		\caption{Samples of triangular (a), distorted quadrilateral (b), and hexagonal (c) meshes employed for the numerical tests in this section.}
		\label{fig:meshes}
	\end{figure}
	
	\subsection{Verification of spatial convergence}
	First we consider a steady version of the poroelasticity equations. An exact solution of the problem on the square domain $(0,1)^2$ is given by the smooth functions
	\begin{gather*}
	\bu(x,y) = \begin{pmatrix}
	-\cos(2 \pi x) \;\sin(2 \pi y) + \sin(2 \pi y) + \sin^2(\pi x) \sin^2(\pi y) \\
	\sin(2 \pi x ) \cos(2 \pi y) - \sin(2 \pi x)
	\end{pmatrix}, \\
	p(x,y) = \sin^2(\pi x) \sin^2(\pi y), \quad \psi(x,y) = \alpha p - \lambda \vdiv\bu.
	\end{gather*}
	The body load~$\ff$ and the fluid source~$\ell$ are computed by evaluating  these closed-form solutions and the problem is completely characterised after specifying the model constants
	\begin{gather*}	
	\nu=0.3, \quad E_c=100, \quad \kappa=1, \quad \alpha=1, \quad c_0=1, \\
	\eta=0.1, \quad
	\lambda=\frac{E_c \nu }{(1+\nu)(1-2 \nu)}, \quad \mu=\frac{E_c}{(2+2 \nu)}.\end{gather*}
	
	On a sequence of successively refined grids (we have
	employed for this particular case, uniform triangular meshes as depicted in Figure~\ref{fig:meshes}(a))
	we compute errors and convergence rates according to
	the meshsize and tabulating also the number of degrees of freedom (Ndof).
	The experimental error decay (with respect to mesh refinement) is measured using individual
	relative norms defined as follows:
	\begin{gather*}
	e_1(\bu) := \frac{\bigl(\sum_{K \in \mathcal{T}_h} | \bu - \bPi^{\beps}_K \bu_h|_{1,K}^2 \bigr)^{1/2}}{|\bu|_{1,\Omega}}, \quad
	e_0(\bu) := \frac{\bigl(\sum_{K \in \mathcal{T}_h} \| \bu - \bPi^{\beps}_K \bu_h\|_{0,K}^2 \bigr)^{1/2}}{\|\bu\|_{0,\Omega}},\\
	e_1(p) :=  \frac{\bigl(\sum_{K \in \mathcal{T}_h} | p - \Pi^{\nabla}_K p_h|_{1,K}^2 \bigr)^{1/2}}{|p|_{1,\Omega}}, \quad
	e_0(p) := \frac{\bigl(\sum_{K \in \mathcal{T}_h} \| p - \Pi^{\nabla}_K p_h\|_{0,K}^2 \bigr)^{1/2}}{\|p\|_{0,\Omega}}, \\
	e_0(\psi) := \frac{\bigl(\sum_{K \in \mathcal{T}_h} \| \psi - \psi_h\|_{0,K}^2 \bigr)^{1/2}}{\|\psi\|_{0,\Omega}}.
	\end{gather*}

	Table \ref{tab:eg1_err} shows this convergence history, exhibiting optimal error decay.
	
	\begin{table}[t!]
	\setlength{\tabcolsep}{4pt}
	\begin{center}  
	\begin{tabular}{|cccccccccccc|}
				\hline 	
				\text{Ndof} & $h$ & $e_1(\bu)$ & $r$ & $e_0(\bu)$ & $r$ & $e_0(\psi)$ & $r$ & $e_1(p)$ & $r$ & $e_0(p)$ & $r$ \\ 
				\hline
				\hline
				179 & 0.25  & 0.477968 & - & 0.271687 & - &   0.508386 & - & 0.444463 & - & 0.142539 & -\\
				819 & 0.125 & 0.204990 & 1.22 & 0.055766 & 2.28 & 0.198845 & 1.35 & 0.195632 & 1.18 & 0.029745 & 2.26\\
				3419 & 0.0625 & 0.097838 & 1.07 & 0.013083 & 2.09 & 0.091837 & 1.11 &  0.097854 & 1.00 & 0.007526 & 1.98 \\
				13819 & 0.03125 & 0.049954 & 0.97 &  0.003322 & 1.98 & 0.043829 & 1.07 & 0.024456 & 1.02 & 0.001842  & 2.03 \\
				56067 & 0.015625 & 0.024756 & 1.01 & $8.2\cdot10^{-4}$ & 2.02 & 0.021704 & 1.01 & 0.024456 & 0.98 & $4.7\cdot10^{-4}$ & 1.96 \\
				\hline
		\end{tabular}
		\end{center}
		\caption{Verification of space convergence for the method with $k=1$. Errors and convergence rates $r$ for solid displacement, total pressure and fluid pressure.}\label{tab:eg1_err}
	\end{table}

	\subsection{Convergence with respect to the time advancing scheme}
	
	Regarding the convergence of the time discretisation, we fix a relatively
	fine hexagonal mesh and construct successively refined partitions of the time
	interval $(0,1]$. As in Ref. \cite{vermaI}, and in order to avoid mixing errors coming from the spatial discretisation, we
	modify the exact solutions to be
	\begin{align*}
	&  \bu(x,y,t)= 100\sin(t)\begin{pmatrix}\frac{x}{\lambda}+{y}, \\
	x+\frac{y}{\lambda}\end{pmatrix}, \\ &  p(x,y,t) = \sin(t)(x+y), \quad \psi(x,y,t) = \alpha p - \lambda \vdiv\bu, \end{align*}
	and we use them to compute loads, sources, initial data, boundary values, and boundary fluxes. The model parameters assume the values
	\begin{equation}\label{eq:param3}
	\kappa=0.1, \quad \alpha=1, \quad c_0=0, \quad \eta=1,
	\quad \lambda=1\times10^3 \quad \mu=1.\end{equation}
	The boundary definition is $\Gamma = [\{0\}\times (0,1)]\cup [(0,1) \times \{0\}]$ (bottom and left edges) and $\Sigma = \partial\Omega\setminus\Gamma$.
	
	We recall that cumulative errors up to $t_{\text{final}}$ associated with solid displacement, fluid pressure, and a generic pressure $v$ (representing either fluid or total pressure), are defined as
	\begin{align}\label{eq:errors-t} \begin{split}
	E_0(\bu) & =  \biggl(\Delta t \sum_{n=1}^N \biggl(
	\sum_{K \in \mathcal{T}_h} \| \bu(t_n) - \bPi^{\beps}_K \bu^n_h\|_{0,K}^2\biggr)
	\biggr)^{1/2},  \\
	{E_0}(v) & =  \biggl(\Delta t \sum_{n=1}^N
	\biggl(\sum_{K \in \mathcal{T}_h} \| v(t_n) - \Pi^{\nabla}_K v^n_h\|_{0,K}^2
	\biggr)\biggr)^{1/2},
	\end{split} \end{align}
	respectively.
	
	From Table~\ref{table:err_time} we can readily observe that these errors
	decay with a rate of $O(\Delta t)$.
	
	\begin{table}[!t]
	\begin{center}
	\begin{tabular}{|ccccccc|}
				\hline
				$\Delta t$ & ${E_0}(\bu)$ & $r$ & ${E_0}(p)$ & $r$ &  ${E_0}(\psi)$ & $r$  \\
				\hline
				\hline
				0.5 & 0.002897  & -- & 0.462768 & -- & 0.398059 & -- \\
				0.25 & 0.001362 & 1.09 & 0.218179 & 1.08 & 0.187834 & 1.08 \\
				0.125 & $6.5173\cdot10^{-4}$ & 1.06 & 0.104546 & 1.06& 0.090044 & 1.06 \\
				0.0625 & $3.1756\cdot10^{-4}$ & 1.04 & 0.050955 & 1.04 & 0.043910 & 1.04 \\
				0.03125 & $1.5664\cdot10^{-4}$ & 1.02 & 0.025123 & 1.02 & 0.021683 & 1.02 \\
				0.015625 & $7.7950\cdot10^{-5}$ & 1.01 & 0.012469 & 1.01 & 0.010826 & 1.00 \\
				\hline
		\end{tabular}\end{center}
		\caption{Convergence of the time discretisation for solid displacement, fluid pressure, and total pressure, using successive partitions of the time interval and a fixed
			hexagonal mesh.}  \label{table:err_time}
	\end{table}
	
	\subsection{Verification of simultaneous space-time convergence for poroelasticity}
	Now we consider exact solid displacement and fluid pressure solving problem \eqref{eq:Biot} on the square domain $\Omega =(0,1)^2$ and on the time interval $(0,1]$, given as
	\begin{align*}
	&  \bu(x,y,t) = \begin{pmatrix} 	-\exp(-t) \sin(2\pi y) (1-\cos(2 \pi x)) + \frac{\exp(-t)}{\mu+\lambda}\sin(\pi x)\sin(\pi y) \\
	\exp(-t)\sin(2 \pi x)(1-\cos(2 \pi y)) + \frac{\exp(-t)}{\mu+\lambda}\sin(\pi x)\sin(\pi y)
	\end{pmatrix} , \\
	& p(x,y,t) = \exp(-t) \sin(\pi x) \sin(\pi y), \quad \psi(x,y,t) = \alpha p - \lambda \vdiv\bu,
	\end{align*}
	which satisfies $\vdiv\bu \to 0$ as $\lambda \to \infty$ (see similar tests in Ref. \cite{fu19,yi17}).
	The load functions, boundary values, and initial data can be obtained from these closed-form solutions, and alternatively to the dilation modulus and permeability
	specified in \eqref{eq:param3}, we here choose larger values $\lambda = 1\times10^4$, and
	$\kappa=1$. 	

	In addition to the errors in \eqref{eq:errors-t}, for displacement and for fluid pressure  we will also compute
	\begin{align*}
	E_1(\bu) &=  \biggl(\Delta t \sum_{n=1}^N \biggl(
	\sum_{K \in \mathcal{T}_h} | \bu(t_n) - \bPi^{\beps}_K \bu^n_h|_{1,K}^2\biggr)
	\biggr)^{1/2}, \\
	E_1(p) & =  \biggl(\Delta t \sum_{n=1}^N
	\biggl(\sum_{K \in \mathcal{T}_h} | p(t_n) - \Pi^{\nabla}_K p^n_h|_{1,K}^2
	\biggr)\biggr)^{1/2}.
	\end{align*}
	We consider here pure Dirichlet boundary conditions for both displacement and fluid pressure.
	A backward Euler time discretisation
	is used, and
	in this case we are using successive refinements of the hexagonal partition  of the domain as shown in Figure \ref{fig:meshes}(c), simultaneously with a successive refinement of the time step. The cumulative errors are again computed until the
	final time  $t=1$, and the results are collected in Table \ref{table:err_eg4}. They show once more  optimal convergence rates for the scheme in its lowest-order form.
	
	Note from this and the previous test, that
	a zero constrained specific storage coefficient
	does not hinder the convergence properties.
	
	\begin{table}[!t]
	\setlength{\tabcolsep}{4pt}
	\begin{center}
	\begin{tabular}{|cccccccccccc|}
							\hline
				$h$ & $\Delta t$ & $E_1(\bu)$ & $r$ & $E_0(\bu)$ & $r$ & $E_1(p)$ & $r$ & $E_0(p)$ & $r$ &  $E_0(\psi)$ & $r$  \\
				\hline
				\hline
				1/8 & 1/10 & 1.741116 & - & 0.101035 & - & 0.239518 & - & 0.009757 & - & 0.509493 & -  \\
				1/16 & 1/20 & 0.892377 & 0.96 & 0.026166 & 1.95 & 0.123684 & 0.95 & 0.002528 & 1.95 & 0.251106 & 1.02  \\		
				1/32 & 1/40 & 0.451402 & 0.98 & 0.006594 & 1.99 & 0.062743 & 0.98 & 0.000642 & 1.98 & 0.125025 & 1.01 \\
				1/64 & 1/80 & 0.227050 & 0.99 & 0.001650 & 2.00 & 0.031584 & 0.99 & 0.000161 & 1.99 & 0.062399 & 1.00 \\
				1/128 & 1/160 &	0.113876 & 1.00 & 0.000413 & 2.00 & 0.015844 & 1.00 & 0.000041 & 2.00 & 0.031165 & 1.00  \\ 		
				\hline
		\end{tabular}\end{center}
		\caption{Convergence of the numerical method for displacement, fluid pressure, and total pressure, up to the final time $t=1$, using simultaneous partitions of the time interval and of the
			spatial domain (using
			hexagonal meshes).}\label{table:err_eg4}
	\end{table}

	\subsection{Gradual compression of a poroelastic block}
	Finally we carry out a test involving the compression of a block occupying the region
	$\Omega = (0,1)^2$ by applying a sinusoidal-in-time traction
	on a small
	region on the top of the box (see a similar test in Ref. \cite{oyarzua16}). The model parameters in this case are
	\begin{gather*}
	\nu=0.49995,\quad E_c=3\times10^{4},\quad \kappa=1\times10^{-4},\quad \alpha=1,\quad c_0=1\times10^{-3}, \\
	\eta=1, \quad
	\lambda=\frac{E_c \nu }{(1+\nu)(1-2 \nu)}, \quad \mu=\frac{E_c}{(2+2 \nu)}.\end{gather*}
	
	For this test we have employed a mesh conformed by distorted quadrilaterals exemplified in
	Figure~\ref{fig:meshes}(b).
	The boundary conditions are of homogeneous Dirichlet type for fluid
	pressure on the whole boundary, and of mixed type for displacement, and the
	boundary is split as $\partial \Omega:= \Gamma_1 \cup \Gamma_2 \cup \Gamma_3$.
	A traction $\boldsymbol{h}(t) =(0,- 1.5\times10^{4}\sin(\pi t))^T$ is applied on a segment of the top edge of the boundary $\Gamma_1=(0.25,0.75) \times \{ 1 \}$, on the
	remainder of the top edge $\Gamma_2= [0,1] \times \{ 1 \}  \backslash \Gamma_1$, we
	impose zero traction, and
	the body is clamped on the
	remainder of the boundary
	$\Gamma_3= \partial \Omega \backslash (\Gamma_1 \cup \Gamma_2)$. No boundary conditions
	are prescribed for the total pressure. Initially the
	system is at rest $\bu(0)=\cero$, $\psi(0) = 0$, $p(0) = 0$, and we employ a backward Euler
	discretisation of the time interval $(0,0.5]$ with a constant timestep $\Delta t = 0.1$.
	The numerical results obtained at the final time are depicted in Figure \ref{fig:footing}, where
	the profiles for fluid and total pressure present no spurious oscillations.
	
	\begin{figure}[t]
		\subfigure[]{\includegraphics[width=0.32\textwidth]{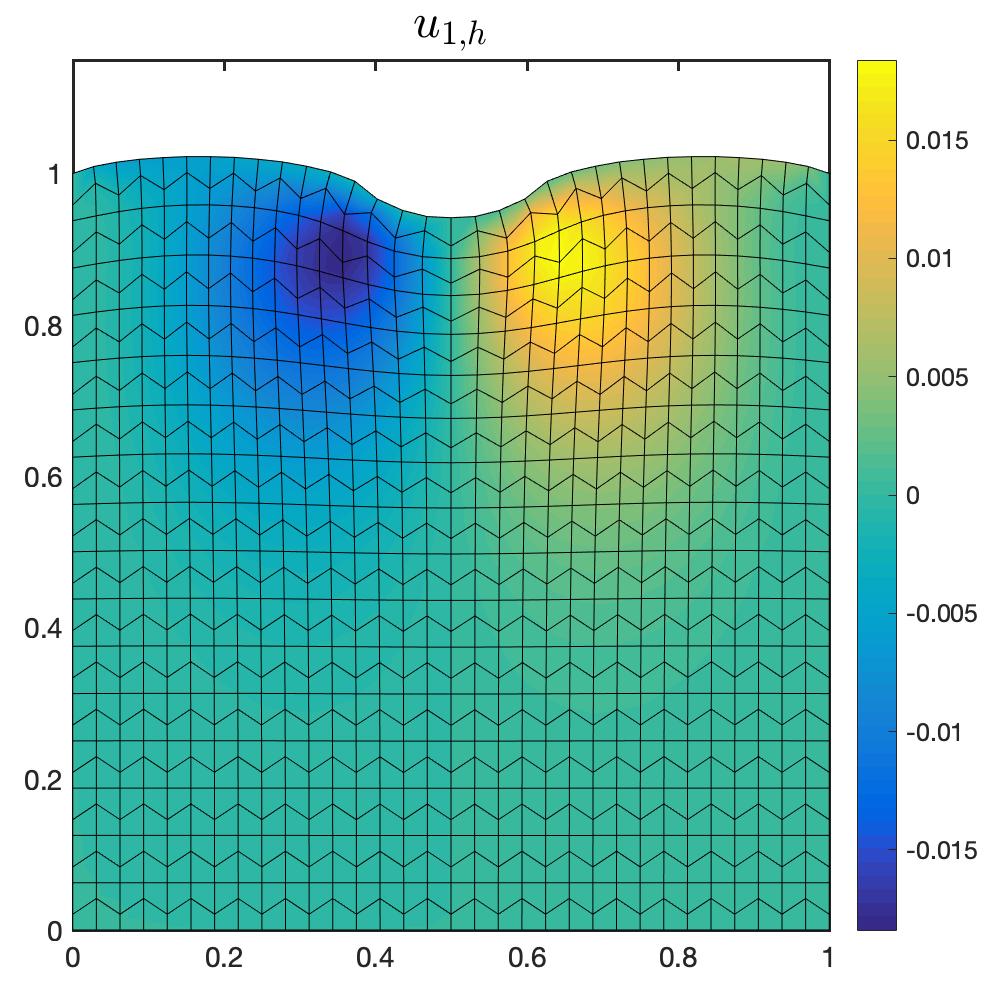}}
		\subfigure[]{\includegraphics[width=0.32\textwidth]{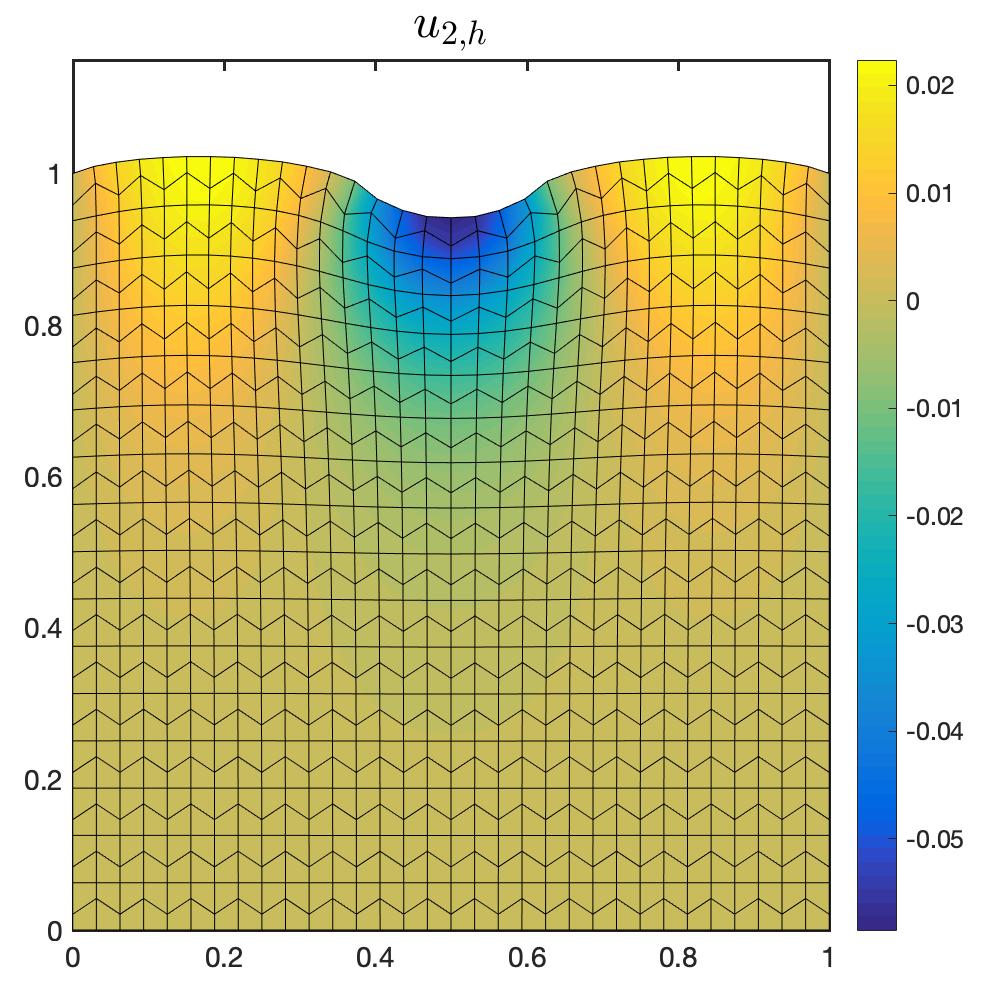}}\quad \raisebox{1mm}{\subfigure[]{\includegraphics[width=0.276\textwidth]{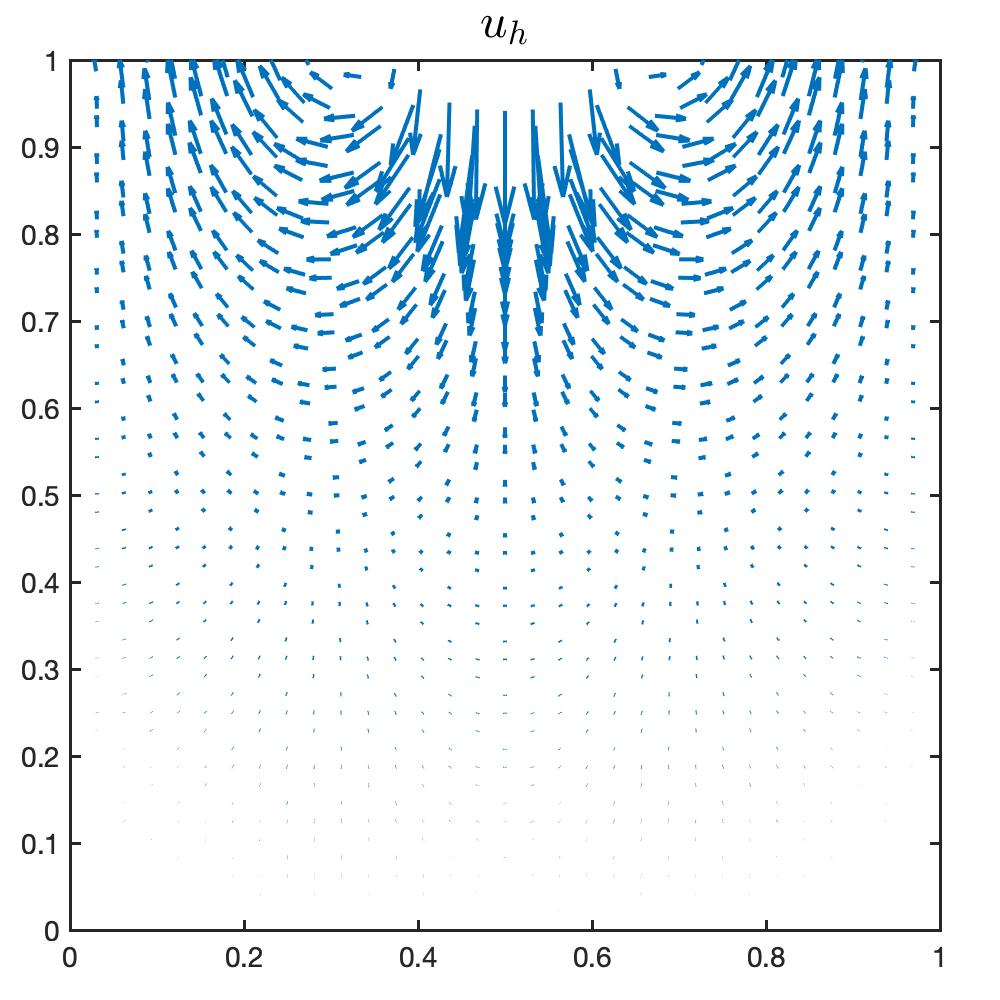}}}\\
		\subfigure[]{\includegraphics[width=0.32\textwidth]{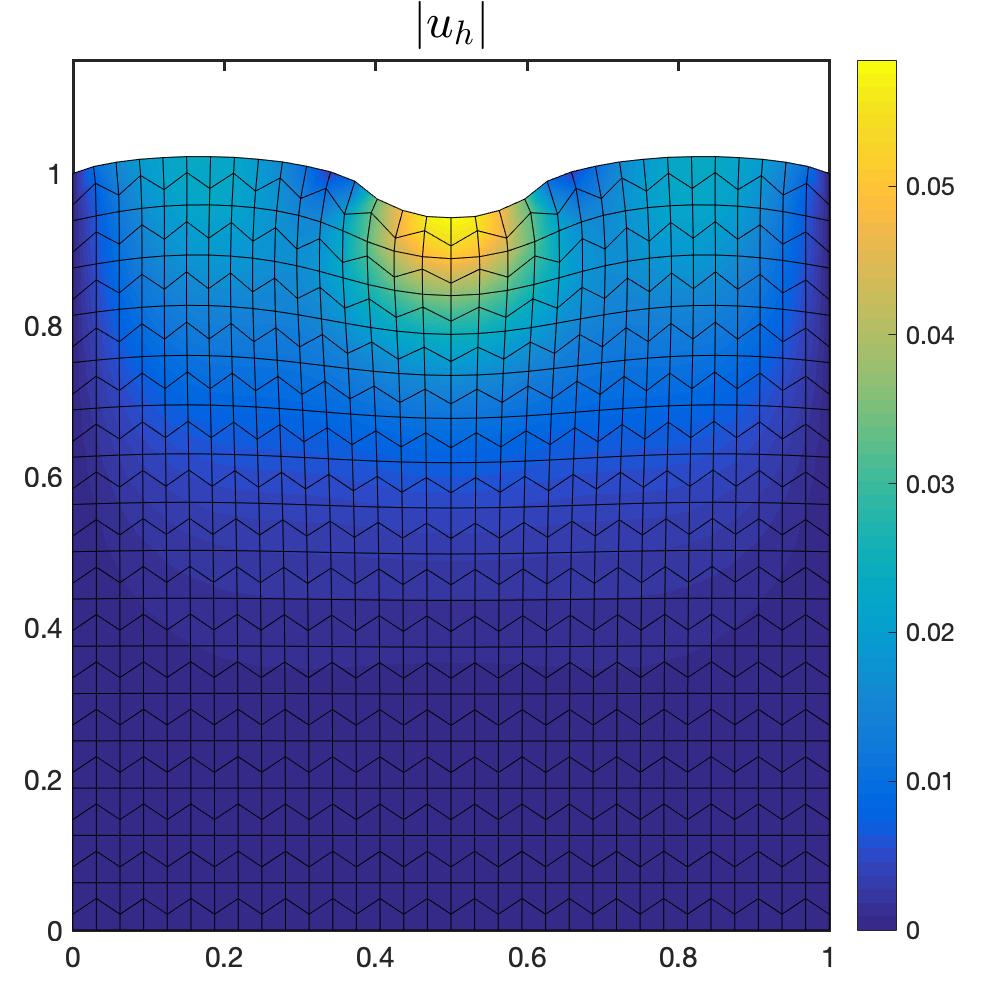}}
		\subfigure[]{\includegraphics[width=0.32\textwidth]{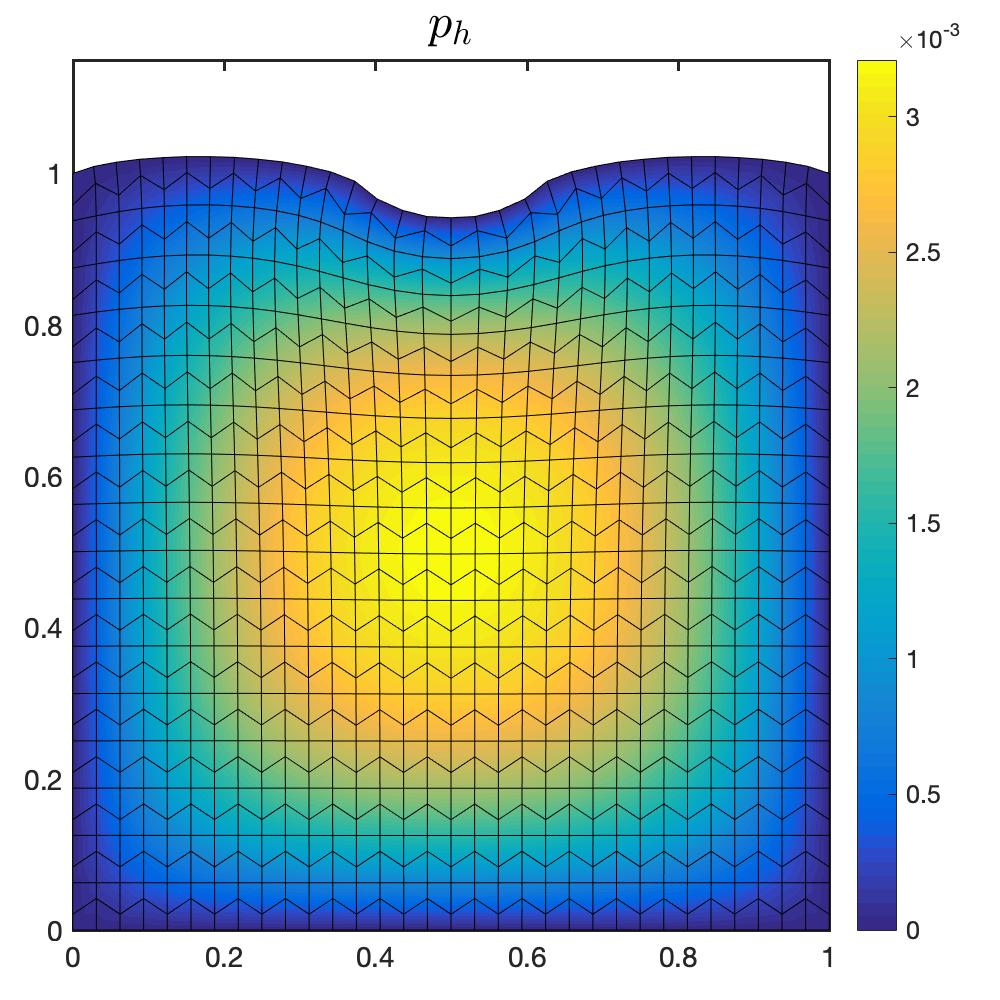} }
		\subfigure[]{\includegraphics[width=0.32\textwidth]{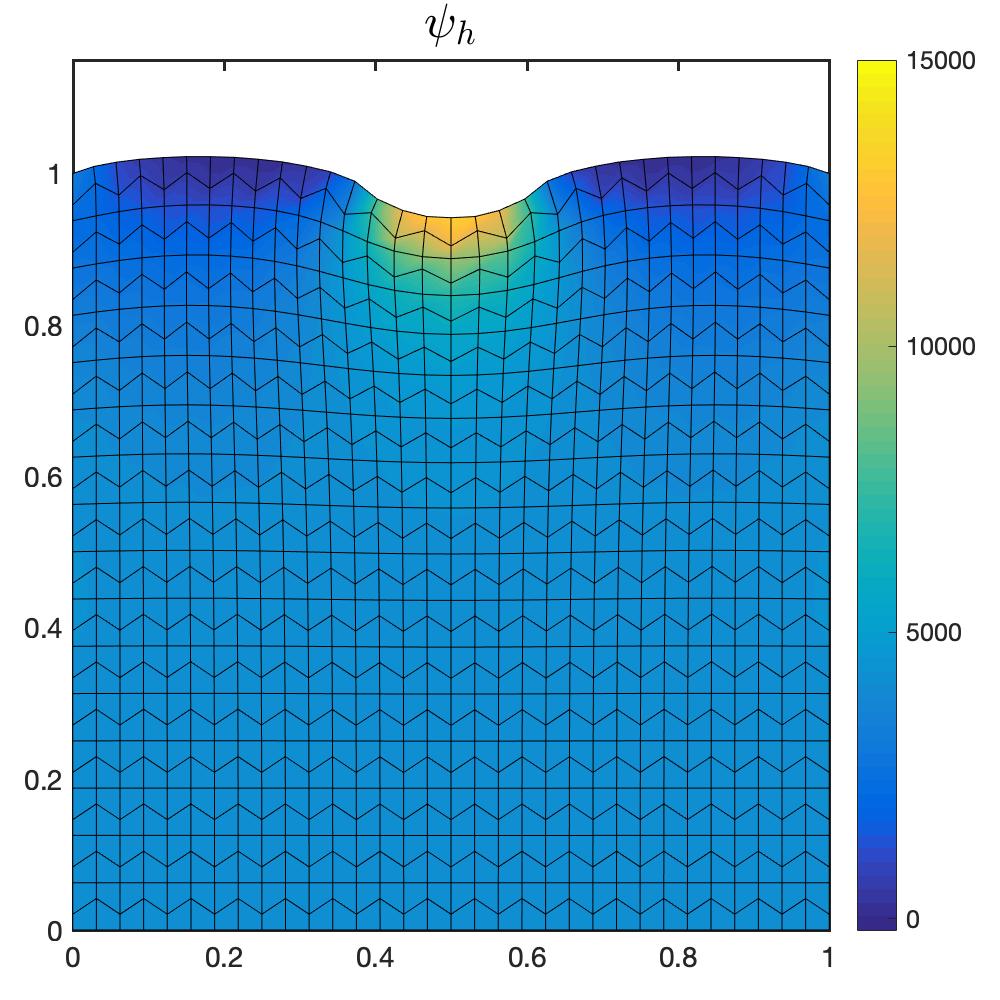}}
		\caption{Compression of a poroelastic block after $t=0.5$ adimensional units.
			Approximate displacement
			components (a,b), displacement vectors on the undeformed domain (c),
			displacement magnitude (d), fluid pressure (e), and
			total pressure (f), depicted on the deformed domain.}
		\label{fig:footing}
	\end{figure}

	\section{Summary and concluding remarks} \label{sec:concl}
	We have constructed and analysed a new virtual element method for the
	Biot equations of linear poroelasticity. The finite-dimensional formulation is based on Bernardi-Raugel type elements, which can be regarded as low-order and stable virtual elements, hence being  computationally competitive compared to other existing stable pairs for incompressible flow problems. Both the formulation and its analysis seem to be novel, and they constitute the first
	fully VEM discretisation for poroelasticity problems.
	
	Optimal and Lam\'e-robust error estimates were established for solid displacement, fluid pressure, and total pressure, in natural norms without weighting. This was achieved with the help of appropriate poroelastic projection operators. Numerical experiments have been performed using different polygonal meshes, and they put into evidence  not only computational verification of the convergence of the scheme (where rates of error decay in space and in time are in excellent agreement with the theoretically derived error bounds), but also its performance in simple poromechanical tests.
	
	Natural extensions of this work include the development and analysis of higher-order versions of the virtual discretisations advanced here, the efficient implementation and application to 3D problems, and the coupling with other phenomena such as diffusion of solutes in poroelastic structures\cite{vermaI}, interface elasticity-poroelasticity problems\cite{adgmr20}, multilayer poromechanics\cite{naumovich06}, or multiple-network consolidation models\cite{lee19,hong19}.
	

	\small 	
	\section*{Acknowledgements} 	
	RB is supported by CONICYT (Chile) through projects Fondecyt~1170473;
	CONICYT/PIA/AFB170001; and CRHIAM, project CONICYT/FONDAP/15130015.
	DM is supported by CONICYT-Chile through FONDECYT project
	1180913 and by project AFB170001 of the PIA Program: Concurso Apoyo
	a Centros Cient\'ificos y Tecnol\'ogicos de Excelencia con Financiamiento Basal.
	RRB is supported by the Engineering and Physical Sciences Research Council (EPSRC) through
	the grant EP/R00207X/1, and by the London Mathematical Society - Scheme 5, grant 51703.
	%

\end{document}